\documentclass[a4paper,11pt]{article}

\usepackage{amsmath, amssymb, amsfonts, amsthm, enumerate, amsthm, hyperref, todonotes, amsrefs}

\usepackage[a4paper,top=3cm,bottom=3cm,left=3cm,right=3cm,marginparwidth=3cm]{geometry} 
\usepackage{tasks}
\usepackage{extarrows} 
\usepackage{mathtools}
\newtheorem{theorem}{\bf Theorem}[section]
\newtheorem{claim}[theorem]{\bf Claim}
\newtheorem{lemma}[theorem]{\bf Lemma}

\newtheorem{conjecture}[theorem]{\bf Conjecture}

\newtheorem{prop}[theorem]{\bf Proposition}
\newtheorem{corollary}[theorem]{\bf Corollary}
\newtheorem{cor}[theorem]{\bf Corollary}

\newtheorem{question}[theorem]{\bf Question}
\newcommand{\secref}[1]{\hyperref[#1]{Section~\ref*{#1}}}
\newcommand{\thmref}[1]{\hyperref[#1]{Theorem~\ref*{#1}}}
\newcommand{\lemref}[1]{\hyperref[#1]{Lemma~\ref*{#1}}}
\newcommand{\propref}[1]{\hyperref[#1]{Proposition~\ref*{#1}}}
\newcommand{\corref}[1]{\hyperref[#1]{Corollary~\ref*{#1}}}
\newcommand{\defref}[1]{\hyperref[#1]{Definition~\ref*{#1}}}
\newcommand{\obsref}[1]{\hyperref[#1]{Observation~\ref*{#1}}}
\newcommand{\conjref}[1]{\hyperref[#1]{Conjecture~\ref*{#1}}}
\newcommand{\claref}[1]{\hyperref[#1]{Claim~\ref*{#1}}}
\newcommand{\figref}[1]{\hyperref[#1]{Figure~\ref*{#1}}}
\let\eps=\varepsilon
\newcommand{\cF}{\mathcal{F}}
\newcommand{\Real}{\mathbb{R}}
\newcommand{\Nat}{\mathbb{N}}
\newcommand{\Z}{\mathbb{Z}}
\DeclareMathOperator{\Stab}{Stab}

\let\Olfloor\lfloor
\let\Orfloor\rfloor
\let\Olceil\lceil
\let\Orceil\rceil
\renewcommand{\lfloor}{\left\Olfloor}
\renewcommand{\rfloor}{\right\Orfloor}
\renewcommand{\lceil}{\left\Olceil}
\renewcommand{\rceil}{\right\Orceil}

\usepackage{comment}

\begin{document}

\title{Degree conditions forcing directed cycles\thanks{This work was partially supported by the National Science Centre grant 2016/21/D/ST1/00998 and the ERC Consolidator Grant LaDIST 648509.}}
\author{
Andrzej Grzesik\thanks{Faculty of Mathematics and Computer Science, Jagiellonian University, ul.~Prof.~St.~\L ojasiewicza 6, 30-348 Krak\'{o}w, Poland. E-mail: {\tt Andrzej.Grzesik@uj.edu.pl}.}\and
Jan Volec\thanks{Department of Mathematics, Faculty of Nuclear Sciences and Physical Engineering, Czech Technical University in Prague, Trojanova 13, 120 00 Prague, Czech Republic. E-mail: {\tt jan@ucw.cz}.}
}

\date{}

\maketitle

\begin{abstract}
Caccetta--H\"aggkvist conjecture is a longstanding open problem on degree conditions that force an oriented graph to contain a directed cycle of a bounded length.
Motivated by this conjecture, Kelly, K\"uhn, and Osthus initiated a study of degree conditions forcing the containment of a directed cycle of a given length.
In particular, they found the optimal minimum semidegree, that is, the smaller of the minimum indegree and the minimum outdegree, which forces a large oriented graph to contain a directed cycle of a given length not divisible by $3$, and conjectured the optimal minimum semidegree for all the other cycles except the directed triangle. 

In this paper, we establish the best possible minimum semidegree that forces a large oriented graph to contain a directed cycle of a given length 
divisible by $3$ yet not equal to~$3$, 
hence fully resolve the conjecture by Kelly, K\"uhn, and Osthus.
We also find an asymptotically optimal semidegree threshold of any cycle with a given orientation of its edges with the sole exception of a directed triangle.
\end{abstract}

\maketitle

\section{Introduction}\label{sec:intro}

One of the most famous open problems in graph theory is the conjecture of Caccetta and H\"aggkvist from 1978.
\begin{conjecture}[Caccetta--H\"aggkvist~\cite{CacH78}]\label{conj:CH}
Every oriented graph $G$ on $n$ vertices with minimum outdegree $\delta^+(G) \ge \frac{n}{\ell}$ contains a directed cycle of length at most $\ell$.
\end{conjecture}

By an oriented graph, we understand a directed graph without loops and multiple edges.
Since the time the conjecture was stated, it attracted a great amount of researchers.
In particular, it was the topic of a 2006 workshop at the American Institute of Mathematics. 
Partial results on \conjref{conj:CH} include its proofs for large values of $\ell$ \cites{CacH78, Ham87, HoaR87, She00}, results with an additive error term in the bound for the cycle length \cites{ChvS83, Nis88, She02} and a multiplicative error term in the minimum outdegree assumption \cites{Bon97, CacH78, HamHK07, HlaKN17, LiaX12, She98}, as well as solutions with an additional assumption on forbidden subgraphs \cites{Grz17,Raz13}. 
For more results and problems related to the Caccetta--H\"aggkvist conjecture, see a summary of Sullivan~\cite{Sul06}.

A weaker conjecture, where $\frac{n}{\ell}$ lower bounds the minimum semidegree, that is, the smaller of the minimum outdegree and the minimum indegree, is also widely open and is very closely related to a conjecture of Behzad, Chartrand, and Wall~\cite{BehCW70} from 1970 on directed cages, that is, directed graphs that are regular and have a prescribed girth.

Motivated by these open problems, Kelly, K\"uhn, and Osthus \cite{KelKO10} conjectured the minimum semidegree guaranteeing existence of a directed cycle of a given length.

\begin{conjecture}[Kelly--K\"uhn--Osthus \cite{KelKO10}]\label{conj:main}
For every $\ell \ge 4$ there exists $n_0:=n_0(\ell)$ such that every oriented graph $G$ on $n\ge n_0$ vertices with $\delta^\pm(G) \ge \frac{n}{k} + \frac{1}{k}$ contains a directed cycle of length exactly~$\ell$, 
where $k$ is the smallest integer greater than $2$ that does not divide~$\ell$.
\end{conjecture}

Observe that the conjectured semidegree threshold, if true, would be the best possible.
Indeed, consider a so-called balanced blow-up of $C_k$, that is, an oriented graph on $n = b\cdot k$ vertices split into $k$ cyclically ordered parts each of size~$b$, where the edges go from all the vertices of a given part to all the vertices of the next part.
The considered graph has no $\ell$-cycles and both the indegree and the outdegree of all its vertices equals to $b=\frac{n}{k}$.

In their paper, Kelly, K\"uhn and Osthus proved \conjref{conj:main} for $k=3$, that is, when the cycle length is not divisible by~$3$.
Moreover, they showed that the corresponding bound $(n+1)/3$ on the minimum semidegree force any cycle of a fixed length $\ell \ge 4$.
\begin{theorem}[\cite{KelKO10}]\label{thm:KelKO}
For every $\ell \ge 4$ there exists $n_0:=n_0(\ell)$ such that every oriented graph $G$ on $n\ge n_0$ vertices with $\delta^\pm(G) \ge \frac{n}{3} + \frac{1}{3}$ contains a directed cycle of length exactly~$\ell$.
\end{theorem}
Note that if \conjref{conj:CH} is true for $\ell = 3$, then the bound $\frac{n+1}3$ on the minimum semidegree of a sufficiently large $n$-vertex oriented graph guarantees existence of any directed cycle of a fixed length.
In 1978, Thomassen~\cite{Tho81} actually conjectured that $n/3$ should also be the semidegree threshold for containing a Hamilton cycle, that is, a cycle going through all the vertices.
However, this was disproved about 15 years later by H\"aggkivst \cite{Hag93} who presented non-Hamiltonian $n$-vertex graphs with the minimum semidegree $\lceil \frac{3n-4}8 \rceil - 1$.
Building on an asymptotic solution of Kelly, K\"uhn, and Osthus~\cite{KelKO08}, Keevash, K\"uhn, and Osthus~\cite{KeeKO09} proved that any sufficiently large $n$-vertex graph with the minimum semidegree $\lceil \frac{3n-4}8 \rceil$ contains a Hamilton cycle.
In fact, the results in~\cites{KeeKO09, KelKO10} yield that $\lceil \frac{3n-4}8 \rceil$ is the semidegree threshold for an oriented graph on $n$ vertices to be \emph{pancyclic}, that is, contain directed cycles of every length between $3$ and $n$.

One of the natural ways to tackle \conjref{conj:main} is to consider its asymptotic relaxation, that is, adding an extra $o(n)$ term to the minimum semidegree assumption.
In~\cite{KelKO10}, Kelly, K\"uhn, and Osthus proved \conjref{conj:main} asymptotically for $k \in \{4,5\}$ with an additional assumption on the cycle lengths. Specifically, they assumed $\ell \ge 42$ and $\ell\ge2556$ when $k=4$ and $k=5$, respectively.
This was later extended by K\"uhn, Osthus, and Piguet~\cite{KuhOP13} to an asymptotic version of \conjref{conj:main} for any $k \ge 7$ and $\ell \ge 10^7k^6$.

\subsection{Our results}

In this paper, we determine the exact value of the semidegree threshold of every directed cycle of length $\ell \ge 4$ in large enough oriented graphs.
\begin{theorem}\label{thm:main}
Fix an integer $\ell \ge 4$ and let $k$ be the smallest integer greater than $2$ that does not divide $\ell$.
There exists $n_0:=n_0(\ell)$ such that the following is true for every oriented graph $G$ on $n\ge n_0$ vertices: 
\begin{itemize}
\item if $\ell \not\equiv 3$ mod $12$ and $\delta^\pm(G) \ge \frac{n}{k}+\frac{k-1}{2k}$, then $G$ contains $C_\ell$, and 
\item if $\ell \equiv 3$ mod $12$ and $\delta^\pm(G) \ge \frac{n}{4}+\frac{1}{4}$, then $G$ contains $C_\ell$.
\end{itemize}
Moreover, for every $\ell\ge4$ the stated bound is the best possible.
\end{theorem}
In particular, the threshold conjectured in~\cite{KelKO10} is correct only when either $\ell$ is not divisible by $3$, or $\ell \equiv 3$ mod $12$.

When $k\ge5$, that is, when $\ell$ is divisible by $12$, we also prove a stability-type result.
\begin{theorem}\label{thm:kbig_walk}
Fix an integer $\ell$ that is divisible by $12$ and let $k$ be the smallest integer greater than~$2$ that does not divide $\ell$.
There exists $n_0:=n_0(\ell)$ such that every oriented graph $H$ on $n \ge n_0$ vertices with $\delta^\pm(H) \ge \frac{n}{k}\left(1-\frac{1}{30k}\right)$ that contains no closed walk of length~$\ell$ is homomorphic to $C_{k}$.
\end{theorem}

Observe now that \thmref{thm:kbig_walk} cannot hold in this form when $\ell=3a$ for odd integers~$a$.
Indeed, any balanced orientation of $K_{n/2,n/2}$ yields an oriented graph $G$ with semidegree $\delta^\pm(G) = \frac{n}{4}$ yet every closed walk in $G$ has an even length.
Nevertheless, when $k=4$ and $\ell\ge9$, we prove the following stability-type result.

\begin{theorem}\label{thm:k4walk}
Fix an integer $\ell \ge 9$ that is divisible by $3$ and not divisible by $4$.
Every oriented graph $H$ on $n$ vertices with $\delta^\pm(H) \ge \frac{n}{4}-o(n)$ that does not contain a closed walk of length~$\ell$ can be made bipartite by removing $o(n)$ vertices.
\end{theorem}

The case of $C_6$ is even more exceptional.
Consider a balanced blow-up of $C_4$ and place inside one of the blobs an arbitrary \emph{one-way oriented} bipartite graph $B$, that is, a bipartite graph with all its edges going from one part to the other one.
Clearly, the obtained graph is $C_6$-free.
However, if $B$ is a one-way oriented $K_{\frac{n}{8},\frac{n}{8}}$ then one must remove $\frac{n}{8}$ vertices to make the resulting graph bipartite.

This paper is organized as follows.
\secref{sec:prelim} introduces the notation and tools we use.
In \secref{sec:constructions}, we construct oriented graphs with large minimum semidegree showing that the threshold in \thmref{thm:main} is the best possible.
Most of the paper is devoted to the proof of \thmref{thm:main}.
Based on the discussion above, our presentation is divided into three more or less independent parts:
In \secref{sec:kbig}, we focus on the cycle lengths~$\ell$ with $k \ge 5$,
\secref{sec:k4} deals with the case $\ell \ge 9$ and $k = 4$,
and we devote \secref{sec:l6} to the semidegree threshold of $C_6$.
Finally, \secref{sec:otherorientations} is concerned with asymptotics of the semidegree thresholds for cycles with a fixed non-cyclic orientation, and \secref{sec:outro} concludes the paper by discussing possible future directions.

\section{Notation and Preliminaries}\label{sec:prelim}

Unless explicitly stated, all the graphs considered in this paper are \emph{oriented graphs}, that is, directed graphs without loops and multiple edges (regardless of the direction).
Given a graph $G$, we write $V(G)$ and $E(G)$ to denote its vertex-set and its edge-set, respectively.
For two vertices $u \in V(G)$ and $w \in V(G)$, we write $uw$ to denote the edge oriented from $u$ to $w$.
Note that if $uw \in E(G)$ then $wu \notin E(G)$, and vice versa.

For a graph $G$ and a vertex $v \in V(G)$, we denote by $N_G^+(v)$ and $N_G^-(v)$ the \emph{outneighborhood} and the \emph{inneighborhood} of $v$, respectively.
If the graph $G$ is known from the context, we omit the subscript and simply write $N^+(v)$ and $N^-(v)$.
More generally, for $S \subseteq V(G)$, we define $N^+(S):=\bigcup_{v \in S} N^+(v)$ and $N^-(S):=\bigcup_{v \in S} N^-(v)$.

The sizes of the outneighborhood and the inneighborhood of a vertex $v$ are referred as the \emph{outdegree} and the \emph{indegree} of $v$,
and we write $\delta^+(G)$ and $\delta^-(G)$ to denote the minimum outdegree and the minimum indegree over all the vertices of~$G$, respectively.
By the \emph{minimum semidegree} $\delta^\pm(G)$ we mean the minimum of $\delta^+(G)$ and $\delta^-(G)$.

Whenever we write a \emph{path} or a \emph{walk} of length $\ell$, we always mean a directed path or a directed walk with $\ell$~edges.
For brevity, we write $C_\ell$ or $\ell$-cycle to denote a directed cycle of length $\ell$, and $P_\ell$ to denote a directed path of length $\ell$.
Given an integer $s \ge2$, the \hbox{\emph{$s$-shortcut}} is a path of length $s$ with an additional edge from the first vertex to the last one, that is, the graph obtained from $C_{s+1}$ after changing the orientation of one edge, see \figref{fig:shortcuts}.
We call its unique vertex of outdegree~$2$ the~\emph{source}, and the unique vertex of indegree~$2$ the~\emph{sink}.
\begin{figure}[ht]
{\hfill
\includegraphics{figures/figure-12}\hfill
\includegraphics{figures/figure-13}\hfill
\includegraphics{figures/figure-14}\hfill
\includegraphics{figures/figure-15}
\hfill}
\caption{The $2$-shortcut, the $3$-shortcut, the $4$-shortcut and a general $s$-shortcut.}\label{fig:shortcuts}
\end{figure}

Throughout the paper, we use the standard Bachmann--Landau notation: for two functions $f$ and $g$ from $\Nat$ to $\Real$, we write $f(n)=o(g(n))$ if $\lim \frac{f(n)}{g(n)} = 0$. 
With a slight abuse of notation, we also write $f(n) \geq o(g(n))$ to denote that $\liminf \frac{f(n)}{g(n)} \geq 0$.

We say that a graph $F$ is \emph{homomorphic} to a graph $H$ if there exists a map $f: V(F) \to V(H)$ that preserves the adjacencies, that is, if $uv \in E(F)$ then $f(u)f(v) \in E(H)$.
A \emph{blow-up} of a graph $H$ is a graph obtained from $H$ by placing an independent set $I_v$ for every $v \in V(H)$, and, for every $uv \in E(H)$, connecting $x$ to $y$ for all $x\in I_u$ and $y \in I_v$.
In case $|I_v|=b$ for every $v \in V(G)$, we call the obtained graph a \emph{$b$-blow-up} of $G$, or a \emph{balanced blow-up} of $G$.
Note that a graph $F$ is homomorphic to $H$ if and only if $F$ is a subgraph of some blow-up of $H$.

For a fixed graph $F$, we say that a graph $G$ is $F$-free if $G$ has no (not necessarily induced) subgraph isomorphic to $F$.
More generally, if $\cF$ is a fixed family of graphs, we say that $G$ is $\cF$-free if $G$ is $F$-free for every $F \in \cF$.

One of the most important and powerful results in graph theory is the \emph{Regularity Lemma} introduced in~1976 by Szemer\'edi~\cite{Sem78}; see also~\cites{KSSS02,RoS10} for excellent surveys on the topic.
A well known application of the regularity method is the so-called \emph{Removal Lemma}~\cite{EFR86}.
Here we state its degree-form for oriented graphs, which can be derived, for example, using the regularity lemma for directed graphs of Alon and Shapira~\cite{AloS04}.
\begin{prop}\label{prop:removallemma}
Fix an $\ell$-vertex oriented graph $F$.
If an $n$-vertex oriented graph $G$ contains $o(n^\ell)$ copies of~$F$, then $G$ contains an $F$-free subgraph $H$ with $|E(G)| - o(n^2)$ edges.
Moreover, one can remove $o(n)$ vertices from~$H$ to get an $F$-free $H' \subseteq G$ with $\delta^\pm(H') = \delta^\pm(G) - o(n)$.
\end{prop}

\subsection{Caccetta--H\"aggkvist conjecture}
In order to prove \thmref{thm:main}, we will make use of some results on Caccetta--H\"aggkvist conjecture.
As we stated in the introduction, the conjecture has been extensively studied and various relaxations has been considered.
A~particular interest is given to the so-called triangle case of the conjecture, that is, when $\ell=3$.

The triangle case of the conjecture states that every oriented graph $G$ on $n$ vertices with $\delta^+(G) \ge n/3$ must contain a directed triangle.
The original paper of Caccetta--H\"aggkvist~\cite{CacH78} established a weaker bound $\delta^+(G) \ge 0.3820n$ which forces an $n$-vertex oriented graph $G$ to contain $C_3$.
After a series of improvements due to Bondy~\cite{Bon97}, Shen~\cite{She98} and Hamburger, Haxell and Kostochka~\cite{HamHK07} as well as
Hladk\'y, Kr\'al' and Norin~\cite{HlaKN17} applied the flag algebra method of Razborov~\cite{Raz07} and established the following:
\begin{theorem}[\cite{HlaKN17}]\label{thm:CH-C3}
Every $n$-vertex oriented graph $G$ with $\delta^+(G) \ge 0.3465n$ contains~$C_3$.
\end{theorem}
Note that De Joannis de Verclos, Sereni and the second author~\cite{JSV} established an improvement to $0.3388n$; however, for our purposes the bound from~\cite{HlaKN17} will be sufficient. 

Following an idea of Shen~\cite{She02}*{Theorem 1}, we use \thmref{thm:CH-C3} to obtain a bound for longer cycle lengths in \conjref{conj:CH}.
The obtained bound is rather weak, nevertheless it will be sufficient for our purposes. 
\begin{lemma}\label{lem:girth}
Every $n$-vertex oriented graph $G$ with $\delta^+(G) \ge (1 - 0.6535^{1/m})n$ contains a directed cycle of length at most $3m$.
\end{lemma}

\begin{proof}
Let $c := 1 - 0.6535^{1/m}$.
For the sake of a contradiction, let $G$ be a counterexample with the smallest number of vertices $n$.
We have that $\delta^+(G) \ge cn$, $G$ contains no directed cycle of length at most~$3m$, but any graph $G'$ on $n' < n$ vertices with $\delta^+(G') \ge cn'$ does contain a cycle of length at most $3m$.

Consider an arbitrary vertex $x\in V(G)$, and let $X_1 := N^+(x)$ and $X_{i+1} := X_i \cup N^+(X_i)$. 
Clearly, for any $1 \le i < m$, the induced graph $G[X_i]$ contains no directed cycle of length at most $3m$.
Since it has less vertices than graph $G$, there exists $x_i \in X_i$ such that $|N^+(x_i) \cap X_i| < c|X_i|$. 
Thus, 
$$|X_{i+1}| > |X_i| + |N^+(x_i)| - c |X_i| \ge cn + (1-c)|X_i|.$$
Since $|X_1|\ge cn$, we obtain that for every $1 < i \le m$,
$$|X_{i}| > (1 - (1-c)^{i})n.$$
In particular,
\begin{equation}\label{eq:ch3m}
|X_m| > (1 - (1-c)^m)n = 0.3465n.
\end{equation}
Therefore, from any vertex of $G$ there is a path of length at most~$m$ to at least $0.3465n$ vertices. 

Consider an auxiliary directed graph $G_m$ on the same set of vertices as $G$ has, where we put an edge from a vertex $u$ to vertex $v$ if $G$ contains a path from $u$ to $v$ of length at most~$m$.
Clearly, if $uv \in E(G_m)$, then $vu \notin E(G_m)$ since $G$ contains no directed cycle of length at most $2m$.
Also, by~\eqref{eq:ch3m}, the graph $G_m$ has the minimum outdegree $\delta^+(G_m) \ge 0.3465n$.
However, \thmref{thm:CH-C3} yields a directed triangle in $G_m$, which translates to a directed cycle of length at most $3m$ in $G$; a contradiction.
\end{proof}

Two specific cases of our proof of \thmref{thm:main} are going to need better bounds than those given by \lemref{lem:girth}.
This will be provided by the following two lemmas, which are straightforward applications of the flag algebra method similar to those used in~\cite{HlaKN17}.

\begin{lemma}\label{lem:flags-k4}
There exists $n_0$ such that there is no $\{C_3,C_6\}$-free oriented graph $G$ on $n \ge n_0$ vertices with $|N^+(v)| = 0.2509n$ for every $v \in V(G)$.
\end{lemma}

\begin{lemma}\label{lem:flags-k5}
There exists $n_0$ such that there is no $\{C_3,C_4,C_6\}$-free oriented graph $G$ on $n \ge n_0$ vertices with $|N^+(v)| = 0.20009n$ for every $v \in V(G)$.
\end{lemma}

The proofs of both lemmas are computer aided, and the verification SAGE-scripts are available at \url{http://honza.ucw.cz/proj/sd_cycle}.
The two lemmas in turn yield the following.

\begin{corollary}\label{cor:flags}
There is $n_0 \in \Nat$ such that the following is true for every oriented graph~$G$ on $n\ge n_0$ vertices.
{\par{(1)}\hspace{0.3em}} If $\delta^+(G) \ge 0.251n$ then $G$ contains $C_6$.
{\par{(2)}\hspace{0.3em}} If $\delta^+(G) \ge 0.2001n$ then $G$ contains $C_{12}$.
\end{corollary}

\begin{proof}
We prove only the first part as the other one is completely analogous.
Suppose there exists an $n$-vertex oriented graph $G$ with $\delta^+(G) \ge 0.251n$ that has no $C_6$.
By the minimum degree variant of Removal Lemma stated in the moreover part of \propref{prop:removallemma}, we find $H \subseteq G$ on $n'$ vertices with $\delta^+(H) \ge 0.2509n'$ that is also $C_3$-free.
Moreover, we may assume that $|N^+(v)|=0.2509n'$ for every $v\in V(H)$ as otherwise we remove edges going from vertices with larger outdegree.

Let $m:=\lceil \frac{n_0}{n'} \rceil$, where $n_0$ is the constant in \lemref{lem:flags-k4}, and let $H_m$ be the $m$-blow-up of $H$.
Clearly, $H_m$ is $\{C_3,C_6\}$-free and every vertex has outdegree $0.2509n'm$. However, this contradicts \lemref{lem:flags-k4}.
\end{proof}

\subsection{Additive number theory}

In order to concatenate different closed walks together, we use Kneser's theorem, a generalization of the classical Cauchy--Davenport theorem. 
\begin{theorem}[Kneser's theorem~\cite{Kne56}, see also \cite{GerR09}]
\label{thm:kneser}
If $A$ and $B$ are nonempty subsets of an additive group $\Z_k$, then 
$$|A + B| \geq |A + \Stab(A+B)| + |B + \Stab(A+B)| - |\Stab(A+B)|,$$ where $\Stab(A+B) = \{x \in \Z_k : x+A+B=A+B\}$.
\end{theorem}

The main usage of Kneser's theorem is the following lemma.

\begin{lemma}\label{lem:additive}
Let $\ell$ be a multiple of $12$, $k\geq5$ the smallest integer that does not divide~$\ell$, and  $(A_1, \ldots, A_m)$ a~collection of $m$ sets of non-zero remainders modulo $k$. If the sum of remainders from $A_1, \ldots, A_m$ (at most one remainder from each set)
cannot be equal to $\ell$~mod~$k$, then there are at least $|A_1|+\ldots+|A_m|$ different non-zero values modulo $k$ obtainable by such sums.
\end{lemma}

\begin{proof}
Let $B_i:=A_i\cup \{0\}$.
Observe that the elements obtainable by summing up some of the remainders from $A_1, \ldots, A_m$ are in one-to-one correspondence with the elements of the sumset $B_1+\ldots+B_m$.
Since $|X+\Stab(X+Y)| \geq |X|$ and $|\Stab(X+Y)| \geq |\Stab(Y)|$ for any two sets $X$ and $Y$, an iterative application of \thmref{thm:kneser} yields that 
\begin{align*}
|B_1+\ldots+B_m| &\ge |B_1|+|B_2+\ldots+B_m| - |\Stab(B_1+\ldots+B_m)| \\
&\ge |B_1|+|B_2| + |B_3+\ldots+B_m| - 2|\Stab(B_1+\ldots+B_m)| \\
&\ge \ldots \ge |B_1|+\ldots+|B_m| - (m-1)|\Stab(B_1+\ldots+B_m)|.
\end{align*}

If $\Stab(B_1+\ldots+B_m)$ is non-trivial, then $k$ cannot be prime and thus $k=p^q$ for some prime number~$p$ and an integer $q\ge2$. Since $\ell$ is divisible by all the integers smaller than~$k$ (so, in particular by $p^{q-1}$), we have $\ell \equiv cp^{q-1}$ mod $k$. Therefore, $\Stab(B_1+\ldots+B_m)$ must contain $\ell$~mod~$k$.
Since $0\in B_1+\ldots+B_m$, the sumset $B_1+\ldots+B_m$ contains $\ell$~mod~$k$ as well.

On the other hand, if $\Stab(B_1+\ldots+B_m)$ is trivial, then we have 
\[|B_1+\ldots+B_m| \ge |B_1|+\ldots+|B_m| - (m-1) = |A_1|+\ldots+|A_m| + 1,\]
and removing $0$ from the sumset $B_1+\ldots + B_m$ yields the desired bound.
\end{proof}

An easy corollary of the above lemma is the following.
\begin{cor}\label{cor:additive}
Let $\ell$ be a multiple of $12$ and $k\geq5$ the smallest integer that does not divide $\ell$. 
If $S$ is a multiset of $k-1$ non-zero reminders modulo $k$, then there exists $S' \subseteq S$ such that 
\[\ell \equiv \sum_{x \in S'} x \mod k.\]
\end{cor}
We note that this corollary is also a consequence of~\cite{AloF88}*{Lemma 4.2}.

\section{Constructions of $C_\ell$-free graphs with the best possible semidegree}\label{sec:constructions}
In this section, we show that the minimum semidegree threshold in \thmref{thm:main} is best possible.
Given a length of the forbidden cycle $\ell$, recall that $k$ denotes the smallest integer greater than $2$ that does not divide $\ell$.
We describe the desired constructions separately for odd and even values of $k$.

\subsection{The value of $k$ is odd}

We shall construct an $n$-vertex graph $G$ with no $C_\ell$ and minimum semidegree equal to $\delta^\pm(G) = \frac{n}{k}+\frac{k-3}{2k}$. 
In particular, $n + (k-3)/2$ must be divisible by $k$. 
Consider first a balanced blow-up of $C_k$ on $n+(k-3)/2$ vertices, which has both the correct minimum semidegree and contains no $C_\ell$. 
Our aim is to remove $(k-3)/2$ vertices from the blow-up with neither changing the minimum semidegree nor creating $C_\ell$.
In order to do so, we use $(k-3)/2$ times one of the three \emph{maneuvers} described in \figref{fig:maneuvers_odd}.

\begin{figure}[ht]
{\hfill
\includegraphics{figures/figure-3} \hfill \includegraphics{figures/figure-0} \hfill \includegraphics{figures/figure-1}
\hfill}
\caption{Maneuvers used when $k$ is odd. In the first one, adding the middle vertex and incident edges, which creates cycles of length $(k+1)/2$ and $(k-1)/2$, allows to remove a vertex in the leftmost blob and in the rightmost blob. In the remaining two maneuvers, adding edges to an arbitrary vertex in each of the two bottommost blobs allows to remove a vertex in the topmost blob.}\label{fig:maneuvers_odd}
\end{figure}

Before applying any of the maneuvers to the blow-up, one can obtain only cycles of lengths congruent to $0$ mod $k$. 
An application of the first maneuver allows us to change the remainder modulo $k$ of an obtainable cycle length by $(k+1)/2$ or $(k-1)/2$.
Similarly, applying the second maneuver allows us to change the remainder of an obtainable cycle length twice by $(k-1)/2$.
Finally, the third maneuver allows us to change the remainder twice by $(k+1)/2$.

Let $r := \ell$ mod $k$, and consider the following $(k-3)/2$ applications of the maneuvers:
\begin{itemize}
\item If $0<r<\frac k4$, then apply $2r-1$ times the first maneuver and $(k-1)/2-2r$ times the second one.
\item If $\frac k4 < r < \frac k2$, then apply $k-1-2r$ times the first maneuver and $2r-(k+1)/2$ times the third one.
\item If $\frac k2 < r < \frac{3k}4$, then apply $2r-k-1$ times the first maneuver and $(3k-1)/2-2r$ times the second one.
\item If $\frac{3k}4 < r < k$, then apply $2k-1-2r$ times the first maneuver and $2r-(3k+1)/2$ times the third one. 
\end{itemize}

Clearly, each application of a maneuver reduces the number of vertices by one, thus the resulting graph has $n$ vertices. We assert that it also does not contain $C_\ell$.
Let us discuss in detail only the case $0<r<\frac k4$ since the other three cases are resolved analogously.
In this case, any obtainable cycle length has the remainder modulo~$k$ equal to $a \cdot \frac{k+1}{2} + b \cdot \frac{k-1}{2}$ for some $a \in \{0, 1, \ldots, 2r-1\}$ and $b \in \{0,1, \ldots, k-2r-2\}$.
In order to obtain the remainder~$r$, the equality $a \cdot \frac{k+1}{2} + b \cdot \frac{k-1}{2} \equiv r$ mod~$k$ must hold. However, this is equivalent to $a-b \equiv 2r$ mod~$k$, which cannot be satisfied.
Therefore, the resulting graph does not contain a cycle of length $\ell$.

\subsection{The value of $k$ is even}

In this case, $k$ must be a power of $2$. 
Observe that if $\ell$ is even then $\ell \equiv \frac{k}{2}$ mod $k$, and
if $\ell$ is odd then $k=4$.
Also note that if $\ell \equiv 3$ mod $4$, then the optimal construction for \thmref{thm:main} is simply a balanced blow-up of~$C_4$. 
For the rest of this section, we will assume $\ell\not\equiv 3$ mod $4$, and our aim will be to construct an $n$-vertex graph without $C_\ell$ that has minimum semidegree $\delta^\pm(G) = \frac{n}{k}+\frac{k-2}{2k}$. 

Analogously to the case when $k$ is odd, we start with a balanced blow-up of $C_k$ on $n+(k-2)/2$ vertices, and apply a sequence of $(k-2)/2$ specific maneuvers.

\begin{figure}[ht]
{\hfill
\includegraphics{figures/figure-2} \hfill \includegraphics{figures/figure-4}
\hfill}
\caption{Maneuvers used when $k$ is even. In the first one, adding the middle vertex and incident edges allows to remove one vertex in each of the two bottommost blobs. In the other one, adding the middle two vertices and incident edges allows to remove a vertex in each blob except the bottom right one.}\label{fig:maneuvers_even}
\end{figure}

Again, a single application of a maneuver decreases the total number of vertices by one.
Now consider the following $(k-2)/2$ applications of the maneuvers based on the two possible relations between $k$ and $\ell$:
\begin{itemize}
\item If $\ell \equiv k/2$ mod $k$, then applying $(k-2)/2$ times the first maneuver yields a graph where any obtainable cycle length modulo $k$ is equal to $m$ or~$-m$ for some \hbox{$m\le (k-2)/2$}. Therefore, the resulting graph contains no $C_\ell$. 

\item If $\ell\equiv 1$ mod $4$, then a single application of the second maneuver yields a graph where every cycle has length $0$, $2$, or $3$ mod $4$. In particular, it does not contain $C_\ell$.
\end{itemize}

\section{Semidegree threshold of $C_\ell$ with $k \ge 5$}\label{sec:kbig}

Throughout the whole section, we assume that $k\ge5$ and $k$ is defined from $\ell$ as in the statement of \thmref{thm:main}.
In particular, $\ell$ is divisible by all the numbers from $3$ to $k-1$ and not divisible by $k$. Also, $k$ needs to be a prime power. 

We start with proving \thmref{thm:kbig_walk}, which is for closed walks instead of cycles, but on the other hand, it gives an additional structural information. 

\begin{proof}[Proof of Theorem~\ref{thm:kbig_walk}]
The proof contains two major claims about graph $H$: one bounds the maximal possible length of a minimal path between any two vertices, and the other one forbids certain subgraphs.
Using the two, we will be able to understand the structure of $H$ so that we can prove the theorem.

Firstly, we use \lemref{lem:girth} to find a short path between any two vertices of $H$. 

\begin{claim}\label{cla:diameter}
For any vertices $x, y \in V(H)$ there exists a directed path from $x$ to $y$ of length at most $10\lfloor\frac{k-1}{3}\rfloor$.
Additionally, for $k=7$, there exists such a path of length at most $12$.
\end{claim}

\begin{proof}
We follow the main idea used in the proof of \lemref{lem:girth}.
Let $m:=\lfloor\frac{k-1}{3}\rfloor$, \hbox{$c:= 1 - 0.6535^{1/m}$}, and fix an arbitrary vertex $x\in V(H)$.
We define $X_1 := N^+(x)$, and $X_{i+1} := X_i \cup N^+(X_i)$ for every $i\in[5m-1]$.

For each $i \in [5m]$, the induced graph $H[X_i]$ cannot contain a directed cycle of length at most $3m \le k-1$, as otherwise $H$ contains a closed walk of length $\ell$.
Therefore, there exists $x_i \in X_i$ such that $|N^+(x_i) \cap X_i| < c|X_i|$ by \lemref{lem:girth}, and 
\[|X_{i+1}| > |X_i| + |N^+(x_i)| - c |X_i| \geq \frac{n}{k}\left(1-\frac{1}{30k}\right) + (1-c)|X_i|.\]
Combining this estimate with $|X_1| \geq \frac{n}{k}(1-\frac{1}{30k})$ yields that
\[|X_{i}| > \frac{n}{k}\left(1-\frac{1}{30k}\right)\frac{1 - (1-c)^{i}}{c} = \frac{n}{k}\left(1-\frac{1}{30k}\right)\frac{1 - 0.6535^{i/m}}{1 - 0.6535^{1/m}}.\]
In particular,
\begin{equation}\label{eq:diameter:X5m}
|X_{5m}| > \frac{n}{k}\left(1-\frac{1}{30k}\right)\frac{1 - 0.6535^5}{1 - 0.6535^{1/m}}.
\end{equation}

Recall that $m=\lfloor\frac{k-1}{3}\rfloor$.
For $k \in \{5,7,9\}$ we get $|X_{5m}| > 0.505n$, $|X_{5m}| > 0.653n$ and $|X_{5m}| > 0.508n$, respectively.
Since~\eqref{eq:diameter:X5m} is ascending in terms of $k$ when considered separately for each reminder modulo~3, we have that $|X_{5m}|>n/2$ for every considered value of $k$.
Note that we used the fact $k \neq 6$ since $k$ must be a prime power.
Therefore, there are paths of length at most $5\lfloor\frac{k-1}{3}\rfloor$ from $x$ to more than half of the vertices of the graph~$H$. 

Now, fix $y \in V(H)$ arbitrarily.
Using the bound on the minimum indegree, we analogously find paths of length at most $5\lfloor\frac{k-1}{3}\rfloor$ to $y$ from more than half of the vertices of $H$.
Therefore, there is a path from $x$ to $y$ of length at most $10\lfloor\frac{k-1}{3}\rfloor$. 

In order to show the additional part of the statement, we notice that in the case $k=7$ we have $|X_{i}| > 0.534n > n/2$ already when $i=6$ (note that $5m=10$ in that case).
This improvement yields a path of length at most $12$ from any $x \in V(H)$ to any $y \in V(H)$.
\end{proof}
Note that we established a bound valid also for small values of $k$, which is highly overestimating for big~$k$.
In particular, one can obtain better bound (below $1.6k$) assuming $k$ is large enough.
Nevertheless, for our purposes, the bound proved in \claref{cla:diameter} is going to be sufficient. 

The next major step in the proof is forbidding shortcuts of a limited length.
Recall that an $s$-shortcut is a directed path of length $s$ with an additional edge from its first vertex to the last one.
\begin{claim}\label{cla:shortcuts}
$H$ does not contain an $s$-shortcut for any $s \le \lfloor\frac{k}{2}\rfloor +1$.
\end{claim}

\begin{proof}
Since the cases $k=5$ and $k=7$ need a more careful analysis, we split the proof into three parts.

\smallskip\noindent
\textbf{Case $k \ge 8$.}
\smallskip

Suppose for contradiction that the graph $H$ contains an $s$-shortcut with $s \le \lfloor\frac{k}{2}\rfloor +1$, and let $u$ and $v$ be its source and its sink vertex, respectively.
From \claref{cla:diameter}, there is a directed path $P$ of length at most $10\lfloor\frac{k-1}{3}\rfloor$ from $v$ to $u$.
Together with the vertices of the $s$-shortcut, we obtain two closed walks of length $|P|+1$ and $|P|+s$. 

Sylvester's solution~\cite{Sylv1882} of the well-known Frobenius Coin Problem asserts that if an integral multiple of $\gcd(a,b)$ cannot be obtained as an integer-weighted sum of positive integers $a$ and $b$, then it is at most $(a-1)(b-1)-1$.
Since \[\gcd(|P|+1, |P|+s) \le s-1 < k-1,\] we conclude that $\ell$ is divisible by $\gcd(|P|+1, |P|+s)$. 
We also have
\[|P| \cdot (|P|+s-1) \le \frac{10}{3}(k-1) \cdot \left(\frac{10}{3}(k-1)+ \frac{k}{2}+1\right) < 13k^2.\]
Therefore, $\ell < 13k^2$ as otherwise $H$ would contain a closed walk of length $\ell$.

We know that $\ell$ is greater than or equal to the least common multiple of all the integers less than $k$.
It is known (see, e.g., \cite{Nai82}) that this is greater than $2^{k-1}$ for $k\ge 8$.
Since $2^{k-1} > 13k^2$ for $k\ge 12$, we need to address only the cases $k=8$, $k=9$, and $k=11$.
If $k=9$ or $k=11$, then the smallest possible value of $\ell$ is $840 > 20\cdot24$ or $2520 > 30\cdot35$, respectively.
In particular, in both cases, we conclude that $H$ contains a closed walk of length~$\ell$; a contradiction.

It remains to consider the case $k=8$, when $\ell$ is divisible by $420$ and $|P| \le 20$.
However, $|P|$ cannot be equal to $20$ or $19$, because it would create $C_{21}$ or $C_{20}$, respectively, which readily yields a closed walk of length $\ell$.
Therefore, $|P|(|P|+s-1) \le 18\cdot22 = 396 < 420$, contradicting that $H$ contains no closed walk of length $\ell$.

\smallskip\noindent
\textbf{Case $k=7$.}
\smallskip

This is the case of all $\ell$ divisible by $60$ and not divisible by $7$. It is enough to prove it for $\ell=60$. 

Firstly, we will observe that for any $x\in V(H)$ the sets $N^+(x)$ and $N^-(x)$ do not contain directed cycles. By symmetry, it is enough to prove it for $N^+(x)$. Suppose on the contrary that $N^+(x)$ contains a cycle, and so it contains a directed path $P$ of length~6, say from $u$ to $v$. From \claref{cla:diameter}, there is a directed path $R$ of length at most 12 from $v$ to $x$. Notice, that together with vertices from $P$ it creates directed cycles of lengths $|R|+1, |R|+2, \ldots, |R|+7$. In each possible case of the length of $R$, one can find $C_6$, $C_{12}$ or $C_{15}$, in particular a closed walk of length $\ell$, which gives a contradiction. Since the set $N^+(x)$ does not contain directed cycles, it contains a vertex without outneighbors in $N^+(x)$. Similarly, the set $N^-(x)$ contains a vertex without inneighbors in $N^-(x)$.

Knowing this, we can improve our bound in \claref{cla:diameter} and prove that for any vertices $x, y \in V(H)$ there is a directed path of length at most $10$ from $x$ to $y$. 
As previously, we define $X_1 = N^+(x)$ and $X_{i+1} = X_i \cup N^+(X_i)$ for $i \ge 1$. We have $|X_1| \ge \frac{n}{7}(1-\frac{1}{210})$. From the above observation, in $X_1$, there is a vertex $x_1$ such that $N^+(x_1) \cap X_1 = \emptyset$, so $|X_2|\ge \frac{2n}{7}(1-\frac{1}{210})$. 
Now, for each $i \ge 2$, we use \lemref{lem:girth} for $m=2$ for the induced graph $H[X_i]$ obtaining $x_i \in X_i$ such that $|N^+(x_i) \cap X_i| < c|X_i|$, where $c = 1 - \sqrt{0.6535} < 0.1917$. Therefore, $|X_3| > 0.372n$, $|X_4| > 0.4428n$ and $|X_5| > n/2$.
Analogously, as in the end of the proof of \claref{cla:diameter}, it means that there is a directed path of length at most $10$ from $x$ to $y$.

Using this observation for $x=y$, we get that each vertex needs to appear in a cycle of length at most $10$. Since there is no closed walk of length $60$, the only possible such short cycles are of length $7$, $8$, and $9$. Now notice that $60$ can be expressed as a sum of multiples of $11$ and multiples of each of the numbers $7$, $8$ and $9$, and so an existence of $C_{11}$ would create a closed walk of length $60$. Thus, there is no $C_{11}$ in $H$.

We can now finish the proof as in the case $k \ge 8$. Consider an $s$-shortcut for $s \le 4$ with $u$ being the source vertex and $v$ being the sink vertex. From the above observations there is a directed path $P$ of length at most $10$ from $v$ to $u$. Together with the vertices of the $s$-shortcut, we obtain two closed walks of length $|P|+1$ and $|P|+s$. In each possible case of the length of $P$ and $s \le 4$ we have a closed walk of length $60$ or $C_{11}$, which is forbidden. 

\smallskip\noindent
\textbf{Case $k=5$.}
\smallskip

This is the case of all $\ell$ divisible by $12$ and not divisible by $5$. It is enough to prove it for $\ell=12$. We follow the lines of the proof for the case $k=7$.

Firstly, we will prove that for any $x\in V(H)$ the sets $N^+(x)$ and $N^-(x)$ do not contain directed cycles. By symmetry, it is enough to prove it for $N^+(x)$. Suppose on the contrary that $N^+(x)$ contains a cycle, and so it contains a directed walk $P$ of length 5, say from $u$ to $v$ (we allow that $u=v$). From \claref{cla:diameter}, there is a directed path $R$ of length at most 10 from $v$ to $x$. Notice, that together with vertices from $P$ it creates closed walks of lengths $|R|+1, |R|+2, \ldots, |R|+6$. In each possible case of the length of $R$, one can find a closed walk of length $12$, which gives a contradiction. Since the set $N^+(x)$ does not contain directed cycles, it contains a vertex without outneighbors in $N^+(x)$. Similarly, the set $N^-(x)$ contains a vertex without inneighbors in $N^-(x)$.

Knowing this, we can improve our bound in \claref{cla:diameter} and prove that for any vertices $x, y \in V(H)$ there is a directed path of length at most $6$ from $x$ to $y$. 
As previously, we define $X_1 = N^+(x)$ and $X_{i+1} = X_i \cup N^+(X_i)$ for $i \ge 1$. We have $|X_1| \ge \frac{n}{5}(1-\frac{1}{150})$. From the above observation, in $X_1$, there is a vertex $x_1$ such that $N^+(x_1) \cap X_1 = \emptyset$, so $|X_2|\ge \frac{2n}{5}(1-\frac{1}{150})$. 
Now, we use \corref{cor:flags} for the induced graph $H[X_2]$ obtaining $x_2 \in X_2$ such that $|N^+(x_2) \cap X_2| < 0.24|X_2|$. Therefore, $|X_3| > n/2$. Analogously as in the end of the proof of \claref{cla:diameter}, it means that there is a directed path of length at most~$6$ from $x$ to $y$.

Using this observation for $x=y$ we get that each vertex needs to appear in a cycle of length at most $6$. Since there is no closed walk of length $12$, each vertex appears in $C_5$. Thus, there is no $C_7$ in $H$.

We can now finish the proof as in the previous cases. Consider an $s$-shortcut for $s \le 3$ with $u$ being the source vertex and $v$ being the sink vertex. From the above observations there is a directed path $P$ of length at most~$6$ from $v$ to $u$. Together with the vertices of the $s$-shortcut, we obtain two closed walks of length $|P|+1$ and $|P|+s$. In each possible case of the length of $P$ and $s \le 3$ we have a closed walk of length $12$ or $C_7$, which is forbidden. 
\end{proof}

We are now ready to analyze the structure of graph $H$ and finish the proof of \thmref{thm:kbig_walk}.

Take any directed path $x_1x_2\ldots x_{k-1}$ of length $k-2$ in $H$ and consider the following $k+1$ sets: for $1\le i \le \lfloor k/2 \rfloor+1$ the sets $N^-(x_i)$, and for $\lfloor k/2 \rfloor \le j \le k-1$ the sets $N^+(x_j)$. The sum of their sizes is above $n$, so they need to have a non-empty intersection. Notice that if for $i$ and $i'$ any two sets $N^-(x_i)$ and $N^-(x_{i'})$ intersect, then we obtain an $s$-shortcut, where $s$ is at most $\lfloor k/2 \rfloor+1$ contradicting \claref{cla:shortcuts}. From the same reason, the sets $N^+(x_j)$ and $N^+(x_{j'})$ cannot intersect for any $j$ and $j'$. If the sets $N^-(x_i)$ and $N^+(x_j)$ intersect, then we obtain a forbidden $2$-shortcut or a directed cycle of length at most $k$. Since in $H$ there are no cycles of length less than $k$, there must be always $C_k$. 

Take such a directed $k$-cycle $y_1y_2\ldots y_k$ and define sets $Y_i=N^+(y_i)$. From now on, we denote them cyclically, in particular, by $Y_{0}$ we understand $Y_k$. Notice that they are disjoint, because otherwise we obtain a shortcut forbidden by \claref{cla:shortcuts}. Let $T$ be the set of the remaining vertices of $H$. We know that 
$$|T| \le n - k\left(\frac{n}{k}\left(1-\frac{1}{30k}\right) \right) = \frac{n}{30k}.$$

For any $y_i$ the set $N^-(y_i)$ must be contained in $Y_{i-2}\cup T$, because otherwise we obtain a cycle of length less than $k$ or a forbidden $2$-shortcut. Moreover, a vertex in $T$ cannot have edges to two different vertices $y_i$ and $y_{i'}$, because it creates a forbidden shortcut. Therefore, we can move from $\bigcup Y_i$ to $T$ at most $n/30k$ vertices to ensure that for all $i$, all vertices of $Y_{i-2}$ have an edge to $y_i$. 

Now, consider any vertex in $Y_i$. It can have edges only to $Y_{i+1}\cup T$ and edges from $Y_{i-1}\cup T$; otherwise, it creates forbidden cycles or shortcuts. In particular, any vertex in $Y_i$ has edges to at least $\frac{n}{k}\left(\frac{14}{15}-\frac{1}{30k}\right) > \frac{13}{15}\frac{n}{k}$ vertices of $Y_{i+1}$. Similarly from $Y_{i-1}$.

We will now describe a procedure that moves one-by-one every vertex $t \in T$ to one of the sets $Y_i$ for some $i \in [k]$,
at each step maintaining the following conditions for every $j \in [k]$ and every vertex $v\in Y_j$:
\begin{itemize}
\item $v$ has edges only to $Y_{j+1} \cup T$ and hence $N^+(v) \cap Y_{j+1} > \frac{13}{15}\frac{n}{k}$, and
\item $v$ has edges only from $Y_{j-1} \cup T$ and hence $N^-(v) \cap Y_{j-1} > \frac{13}{15}\frac{n}{k}$.
\end{itemize}
At the end of this procedure, the set $T$ will be empty so the vertex-partition $(Y_i)$ witnesses that $H$ is a subgraph of a blow-up of $C_k$, that is, $H$ is homomorphic to $C_k$.

Take any vertex $v \in T$.
Since $|T| \le n/15k$, the vertex $v$ needs to have more than $\frac{13}{15}\frac{n}{k}$ outneighbors in $\bigcup Y_i$.
If $v$ has edges to two non-consecutive sets $Y_j$ and $Y_{j'}$ then it creates a forbidden shortcut, so it can have edges only to $Y_j$ and $Y_{j+1}$ for some \hbox{$j \in[k]$}.
Now suppose for contradiction it has edges to both, and let $w \in N^+(v) \cap Y_j$ and \hbox{$u \in N^+(v) \cap Y_{j+1}$}.
$N^+(w) \cap N^+(v) = \emptyset$ because the graph $H$ has no $2$-shortcut. However, \hbox{$|N^+(w) \cap Y_{j+1}| > \frac{13}{15}\frac{n}{k}$}, thus $|N^+(v) \cap Y_{j+1}| < \frac{3}{15}\frac{n}{k}$.
Similarly, $N^-(u) \cap N^+(v) = \emptyset$, but $|N^-(u) \cap Y_j| > \frac{13}{15}\frac{n}{k}$, thus $|N^+(v) \cap Y_j| < \frac{3}{15}\frac{n}{k}$.
In total, we have $|N^+(v)| < \frac{6}{15}\frac{n}{k}$; a contradiction.

By the same reasoning applied to the inneighborhood of $v$, we conclude there are indices $i,o \in [k]$ such that $v$ has inneighbors and outneighbors only inside the sets $Y_i$ and $Y_o$, respectively.
In the case $o=i$, we can find a $3$-shortcut contradicting \claref{cla:shortcuts}.
If $o=i+1$ or $o=i-1$ then there is an edge between $N^+(v)$ and $N^-(v)$ which together with~$v$ creates a $2$-shortcut or a directed triangle.
Finally, in the other cases with $o \neq i+2$, we can find a directed cycle of length less than $k$.
Therefore, $o=i+2$ so the vertex $v$ can be moved to $Y_{i+1}$ without violating the conditions of the procedure.

We showed that the graph $H$ is homomorphic to $C_k$, which finishes the proof of \thmref{thm:kbig_walk}. 
\end{proof}

Using \thmref{thm:kbig_walk}, we are ready to prove \thmref{thm:main} for $k \ge 5$.
Recall that $\ell$ is an integer divisible by all positive integers less than $k$ but not divisible by $k$, and
suppose there exists a $C_\ell$-free graph $G$ on $n$ vertices with $\delta^\pm(G) \ge \frac{n}{k}+\frac{k-1}{2k}$ for infinitely many $n$. 
Since every closed walk of length $\ell$ in $G$ contains at most $\ell-1$ vertices, there are at most $n^{\ell-1}$ closed walks of length $\ell$ in $G$. 
Therefore, by the moreover part of \propref{prop:removallemma}, one can remove $o(n)$ vertices and $o(n^2)$ edges from $G$ in order to find its subgraph $H$ that satisfies the assumptions of \thmref{thm:kbig_walk}.
Thus, the vertices of the initial graph $G$ can be partitioned into sets $T$ and blobs $X_i$ for $1\le i \le k$, where $|T|=o(n)$, $|X_i| = \frac{n}{k} + o(n)$ and each vertex in blob $X_i$ has edges to all but $o(n)$ vertices in $X_{i+1}$ and from all but $o(n)$ vertices in $X_{i-1}$.

Let us refer to the edges inside one blob or between two blobs that do not fulfill the cycle order as \emph{extra} edges. 
Notice that if we have a set of at least $k-1$ disjoint extra edges, then using \corref{cor:additive} we can connect some of them by going around the $k$ blobs and using one extra edge each time (making at most $k+1$ steps between them).
In particular, if $\ell \geq (k-1)(k+2)$ then $G$ contains $C_\ell$.
Thus, unless $\ell = 12$ or $\ell = 24$, there cannot be such a set, and we can move at most $k-2$ vertices from the blobs to $T$ in order to remove all the extra edges. 
In the case $\ell=24$, note that $4$ disjoint edges inside some blob can also be connected to obtain $C_{24}$, so we can move at most $3$ vertices from each blob to $T$ in order to remove all the edges inside blobs.
Now, any set of $3$ disjoint extra edges between the blobs contains extra edges that can be combined by going around the blobs which would again yield $C_{24}$, hence we can move additional at most $2$ vertices from the blobs to $T$ to remove such edges.
We conclude that in case $\ell=24$, there are at most $17$ vertices that are incident to all the extra edges.
Finally, in the remaining case $\ell=12$, note that in any set of $4$ disjoint extra edges from a blob $X_i$ to a blob $X_j$ for some indices $i,j \in [k]$ such that $j \neq i+1$, one can find a subset of extra edges that can be connected by going around the blobs to construct $C_{12}$.
Thus, by moving at most $60$ vertices from blobs to $T$ we can remove all the extra edges. 
We conclude that we may assume that $H$ is an induced subgraph of $G$ on $n-o(n)$ vertices that is homomorphic to $C_k$. 

The proof of \thmref{thm:main} for $k\geq5$ follows from the next lemma. 

\begin{lemma}\label{lem:kbig_sidewalks}
Fix an integer $\ell$ that is divisible by $12$ and let $k$ be the smallest integer greater  than $2$ that does not divide $\ell$. Let $G$ be an $n$-vertex graph with $\delta^\pm(G) \geq \frac{n}{k} + \frac{k-1}{2k}$. If $G$ contains an induced subgraph on $n-o(n)$ vertices that is homomorphic to~$C_k$, then $G$ contains $C_\ell$. 
\end{lemma}

\begin{proof}
Assume by contrary, that there exists an $n$-vertex $C_\ell$-free graph $G$ with semidegree $\delta^\pm(G) \geq \frac{n}{k} + \frac{k-1}{2k}$ for arbitrarily large $n$, and its vertices can be partitioned into set $T:=V(G)\setminus V(H)$ and sets $X_i$ for $1\le i \le k$, where $|T|=o(n)$, $|X_i| = \frac{n}{k} + o(n)$, and each vertex in $X_i$ has edges only to vertices in $X_{i+1} \cup T$ and only from vertices in $X_{i-1} \cup T$.

We refer to the sets $X_1,X_2,\dots,X_k$ as \emph{blobs} and say that a blob $X_i$ is an \emph{in-pointing} blob or an \emph{out-pointing} blob of a vertex $t\in T$ if $|N^-(t) \cap X_i| \ge 2k$ or $|N^+(t) \cap X_i| \ge 2k$, respectively.
Using just edges in $H$ one can obtain only cycles of lengths divisible by $k$.
We will use the semidegree assumption for finding sufficiently many other edges in $G$ that allows us to change the remainder modulo $k$ and construct a cycle of length $\ell$.

A \emph{sidewalk} is a path of length 2 starting and finishing in blobs that has the middle vertex in $T$.
The \emph{value} of a sidewalk is the change in the reminder modulo $k$ that is achieved by this sidewalk.
Formally, if $X_i$ is an in-pointing blob and $X_j$ is an out-pointing blob of a vertex $t\in T$, then $t$ with its corresponding neighbors create a sidewalk of value $i-j+2 \mod k$.
We say that two sidewalks are \emph{compatible} if they are either vertex disjoint or they share exactly one vertex which is then the sink of one of the sidewalks and the source of the other one.
A set of sidewalks is called compatible if and only if there exists a cycle in $G$ that contains all of them. In particular, the sidewalks are pairwise compatible. 

We say that a remainder $r \in [k-1]$ can be \emph{combined} using a set of sidewalks if $r$ can be written as a sum of the sidewalk values modulo $k$, where the sum is taken over a compatible subset of sidewalks.
Our aim is to show that the remainder $\ell$ mod $k$ can be combined using some set of sidewalks.
After we do so, then it only remains to connect the chosen sidewalks to make $C_\ell$.
The main tool we will use in combining the sidewalks is \lemref{lem:additive} and its consequence \corref{cor:additive}.
In our case it means that the remainder $\ell$ mod $k$ can be combined using any compatible set of $k-1$ sidewalks of non-zero values. Let us prove that such a set exists. 

\begin{claim}\label{cla:sidewalks_trashvertex}
If $t\in T$ has $x$ in-pointing blobs and $y$ out-pointing blobs, then we can combine the reminder $\ell$~mod~$k$ or there are at least $x+y-2$ different non-zero values of sidewalks containing~$t$.
\end{claim}

\begin{proof}
As a set $A$ take all the $x$ possible values of sidewalks using one chosen out-pointing blob. And as a set~$B$ take the set of distances in the cyclic order from the chosen out-pointing blob to all the out-pointing blobs (including himself at distance 0). The set of possible values of the sidewalks using $t$ is exactly the set $A+B$. If $\Stab(A+B)$ is non-trivial, 
then $k$ is of the form $p^m$ for some prime $p$ and integer $m$, 
and $p^{m-1}$ is in $\Stab(A + B)$. Since $\ell$ is divisible by $p^{m-1}$, it implies that $\ell$ mod $k$ is in $\Stab(A+B)$, thus one can combine the reminder $\ell$ modulo $k$. Otherwise, from Kneser's theorem (\thmref{thm:kneser}), we get that $|A+B| \ge x+y-1$ and so there are at least $x+y-2$ non-zero values of sidewalks containing $t$. 
\end{proof}

\begin{claim}\label{cla:sidewalks_trash}
There exists a compatible set of at most $k-1$ sidewalks using which one can combine the reminder $\ell$~mod~$k$.
\end{claim}

\begin{proof}
Each vertex in $T$ has less than $4k^2$ neighbors, which are not in out-pointing or in in-pointing blobs, so the total set $X$ of such vertices has size at most $4k^2|T| = o(n)$. 
Let $T' \subseteq T$ be the maximal set of vertices in $T$ creating a compatible set of sidewalks of non-zero values with vertices not from $X$. If $|T'| \geq k-1$, then from \corref{cor:additive} we have the wanted set of sidewalks, thus assume $|T'| < k-1$.  

Let $z$ be the sum of the numbers of in-pointing and out-pointing blobs of the vertices in $T'$. 
Using \claref{cla:sidewalks_trashvertex} for each vertex in $T'$ and \lemref{lem:additive}, we get that using sidewalks containing $T'$ one can combine the remainder $\ell$ mod $k$ or at least $z-2|T'|$ non-zero remainders. If $z-2|T'|\geq k-1$, then we have the remainder $\ell$ mod $k$, so assume that $k-1 > z-2|T'|$. 

Take a set of vertices that contains exactly one vertex from each blob, avoiding the set~$X$. By summing up their indegrees and outdegrees we get at least $2k\delta^\pm(G) \geq 2n + k-1$. The chosen vertices can have at most $z$ neighbors in $T'$, so to or from $T\setminus T'$ they have at least 
$$2n+k-1-z-2(n-|T|) = k-1-z+2|T| > 2(|T|-|T'|)$$ 
edges. This means that there exists a vertex $t \in T\setminus T'$ with at least $3$ neighbors among the chosen vertices. It creates a sidewalk of a non-zero value.
Since each in-pointing or out-pointing blob of $t$ contains at least $2k$ neighbors of $t$ and $|T'|<k-1$, then one can choose to the sought set sidewalks that are using different vertices in blobs for different vertices in $T'\cup\{t\}$. This yields a compatible set of sidewalks, which contradicts the maximality of $T'$. 
\end{proof}

In order to finish the proof of \lemref{lem:kbig_sidewalks}, it remains to prove that the found set of at most $k-1$ compatible sidewalks can be used to find $C_\ell$. 
We can connect the sidewalks by going around the $k$ blobs and using one sidewalk each time. Between the sidewalks we make at most $k+1$ steps, so in total we make at most $(k-1)(k+3)$. It means that if $\ell \geq (k-1)(k+3)$, which is true for any considered here $\ell$ except $12$ and $24$, then we can connect the sidewalks and obtain $C_\ell$. 

In the case $\ell=24$, which means $k=5$, if we have at most $3$ sidewalks then we can connect them to obtain~$C_{24}$, because $24 \geq 3\cdot8$. The only set of $4$ sidewalks that do not contain a smaller subset of sidewalk that can be combined to obtain the remainder $4$, is the set containing $4$ sidewalks o value $1$. One can easily check that independently how such sidewalks will be distributed, one can always connect them in a specific order to create~$C_{24}$.

In the remaining case $\ell=12$, a more detailed analysis is needed, since one can have $4$~compatible sidewalks of value $3$ that cannot be connected to create $C_{12}$. Thus, we focus now on the case $\ell=12$, and so $k=5$ and $\delta^\pm(G) \geq \frac{n+2}{5}$.

We will reduce the set $T$ to a smaller set $T'$ and extend the blobs to $X_i'$ for $i=1,\ldots,5$ by consecutively assigning vertices from $T$ to the appropriate sets based on the adjacencies to the blobs they have. Notice that there are no sidewalks of value $2$, since such a sidewalk can be easily extended to obtain $C_{12}$. Take a vertex $t \in T$ that is not creating a sidewalk of value $1$ and that creates a sidewalk of value $0$. It cannot have inneighbors and outneighbors in the same blob and it has an inneighbor in $X_{i-1}$ and an outneighbor in $X_{i+1}$ for some~$i$. Thus, extend blob $X_i$ by taking $X_i':=X_i\cup\{t\}$. Notice that there are no edges inside extended blobs and there are no edges between the consecutive blobs, that do not fulfill the cyclic order. 
Since $t\in X_i'$ has appropriate neighbors in $X_{i-1}$ and $X_{i+1}$, every vertex in $T\setminus\{t\}$ still cannot create a sidewalk of value $2$ with the extended blobs $X_i'$. Thus, we can continue this procedure and assign all vertices from $T$ that creates sidewalks of value~$0$ and do not create sidewalks of value $1$ to the respective blobs. We end up with a graph that is a blow-up of $C_4$, possibly with some edges between the non-consecutive blobs and with a set $T'$ (possibly empty) of vertices that do not create sidewalks of value~$0$ with the extended blobs, or that create sidewalks of value $1$. 

Notice that if there exist compatible sidewalks of values $3$ and $4$, or two compatible sidewalks of value $1$, then one can easily obtain $C_{12}$. 

Firstly assume that there is a blob of size smaller than $\frac{n+2}{5}-1$, say it is $X_2'$. It implies that every vertex in $X_1'$ is a source of at least two sidewalks of non-zero values and every vertex in $X_3'$ is a sink of at least two sidewalks of non-zero values. Moreover only one vertex in $T'$ (the one possible vertex creating sidewalks of value~$1$) can be counted simultaneously as outneighbor of $X_1'$ and inneighbor of $X_3'$. If we obtained this way compatible sidewalks of values $3$ and $4$, three compatible sidewalks of value $4$, or four compatible sidewalks of value $3$, then we can find $C_{12}$. The only remaining option is to have two compatible sidewalks of value $3$ and one sidewalk of value $1$. But now, if any vertex in $X_1'$ or $X_3'$ is creating more sidewalks of non-zero value, then it needs to be of value $3$ and we can obtain~$C_{12}$. So $X_2'$ needs to be of size at least $\frac{n+2}{5}-2$, be fully connected to $X_1'$ and $X_3'$, and $|T'| \geq 1$. It implies that there is a different blob smaller than $\frac{n+2}{5}$, which gives additional sidewalk of value $3$ that can be combined with the previously proven sidewalks in order to obtain $C_{12}$.

Assume now that every blob is of size at least $\frac{n+2}{5}-1$. If there is no sidewalk of value~$1$, then there are at least two blobs of size smaller than $\frac{n+2}{5}$, say $X_i'$ and $X_j'$. It implies that every vertex in $X_{i-1}'$ and $X_{j-1}'$ is the source of a sidewalk of a non-zero value and every vertex in $X_{i+1}'$ and $X_{j+1}'$ is the sink of a sidewalk of a non-zero value. This gives a set of $4$ compatible sidewalks of value $3$ or value $4$. In both cases we can obtain~$C_{12}$. Thus, assume there is a sidewalk of value $1$. It needs to use a vertex in $T'$ and two consecutive blobs, say $X_1'$ and~$X_2'$. It also implies that there are at least $3$ blobs of size smaller than $\frac{n+2}{5}$. If two of them are neither $X_1'$ nor $X_2'$, then as before we have a set of $4$ compatible sidewalks of values $3$ or $4$, which leads to $C_{12}$. In the remaining case we get that $|X_1'|=|X_2'|=|X_i'|=\frac{n+2}{5}-1$ for some $i \neq 1,2$. It implies that every vertex in $X_{5}'$ and $X_{i-1}'$ is the source of a sidewalk of a non-zero value and every vertex in $X_{3}'$ and $X_{i+1}'$ is the sink of a sidewalk of a non-zero value. This leads to a set of $4$ compatible sidewalks of value $3$ or value $4$, and consequently to~$C_{12}$.
\end{proof}

\section{Semidegree threshold of $C_\ell$ with $k=4$ and $\ell \ge 9$}\label{sec:k4}

In this section we assume that $\ell \ge 9$ and $k=4$, which means that $\ell$ is divisible by $3$ and not divisible by $4$.
We start with proving \thmref{thm:k4walk}, which is a similar stability-type theorem as Theorem~\ref{thm:kbig_walk} proved in the previous section. 

\begin{proof}[Proof of \thmref{thm:k4walk}]
Let $H$ be an $n$-vertex graph with $\delta^\pm(H) \ge \frac{n}{4}-o(n)$ that does not contain a closed walk of length $\ell$. In particular, $H$ is $C_3$-free.
If $H$ does not contain transitive triangles, then the underlying undirected graph is triangle-free and has minimum degree $\frac{n}{2}-o(n)$, so Andr\'asfai--Erd\H{o}s--S\'os Theorem~\cite{AndES74} yields it is bipartite.
Therefore, it is enough to find $o(n)$ vertices in $H$ that covers all its transitive triangles.
Analogously to the previous section, we start with bounding the diameter of the graph~$H$. 

\begin{claim}\label{cla:diam6}
For any vertices $x, y \in V(H)$ there exists a directed path of length at least 2 and at most 6 from $x$ to $y$.
\end{claim}

\begin{proof}
The proof is similar to the proof of \claref{cla:diameter}. Denote $X_1 = N^+(x)$ and iteratively $X_{i+1} = X_i \cup N^+(X_i)$ for $i\geq 1$.
For each $i \ge 1$ the induced graph $H[X_i]$ cannot contain a directed triangle, so from \lemref{lem:girth} there exists $x_i \in X_i$ such that \hbox{$|N^+(x_i) \cap X_i| < 0.3465|X_i|$}. 
Thus, 
$$|X_{i+1}| > |X_i| + |N^+(x_i)| - 0.3465 |X_i|.$$
In particular, since $|X_1| \geq 0.25n-o(n)$, then $|X_2| \geq 0.4133n-o(n)$ and $|X_3| \geq 0.52n-o(n)$.
It means that there is a directed path of length at most $3$ starting in $x$ to more than half of the vertices of the graph~$H$. 

Now, take the vertex $y \in V(H)$ and do analogously for a graph obtained from~$H$ by reversing the edges. This way we find a directed path of length at most $3$ to $y$ from more than half of the vertices of $H$. It means that there is a directed path from $x$ to $y$ of length at least $2$ and at most $6$. 
\end{proof}

Knowing this, we can control the structure of outneighborhood and inneighborhood of every vertex in $H$.

\begin{claim}
\label{cla:almost_acyclic}
For every vertex $x\in V(H)$, there exists a vertex $v$ in $N^+(x)$ satisfying $N^+(v)\cap N^+(x) = \emptyset$ and a vertex $w$ in $N^-(x)$ satisfying $N^-(w)\cap N^-(x) = \emptyset$.
\end{claim}

\begin{proof}
By symmetry, it is enough to prove it for $N^+(x)$, where $x$ is an arbitrary vertex of~$H$.

Assuming the contrary, we get that in $N^+(x)$ there must exists a directed walk of length $4$ from some vertex $a$ to some vertex $b$ (it may happen that $a$=$b$). Consider any directed path of length $\ell - 7$ starting in $b$ terminating in some vertex $c$. This way, we obtained walks from $x$ to $c$ of lengths $\ell-6$, $\ell-5$, $\ell-4$, $\ell-3$ and $\ell-2$.

By \claref{cla:diam6}, there exists an oriented path~$P$ from $c$ to $x$ of length at least 2 and at most~6. In each possible case of the length of $P$, we can obtain a closed walk of length~$\ell$. 
\end{proof}

With this, we can improve our bound in \claref{cla:diam6}. 

\begin{claim}\label{cla:diam5}
For any vertices $x, y \in V(H)$, there exists a directed path of length at least 2 and at most 5 from $x$ to $y$.
\end{claim}

\begin{proof}
From the proof of \claref{cla:diam6}, we know that there exists a directed path of length at most $3$ to at least $0.52n-o(n)$ vertices of the graph $H$.
From \claref{cla:almost_acyclic}, there exists $w \in N^-(y)$ without inneighbors in $N^-(y)$, so there exists a directed path of length at most $2$ to $y$ from at least $0.5n-o(n)$ vertices of $H$.
This means that there is a directed path from $x$ to $y$ of length at least $2$ and at most $5$. 
\end{proof}

\begin{claim}\label{cla:1234}
For any $1\le a \le \ell-5$ there are no vertices $x$ and $y$ in $H$ for which there exist directed paths of lengths $a$, $a+1$, $a+2$, and $a+3$ from $x$ to $y$. In particular, for every vertex $v\in V(H)$, there is no path of length~$3$ in $N^+(v)$ or in $N^-(v)$.
\end{claim}

\begin{proof}
Assume that such vertices exist and consider any directed path of length $\ell - 5 - a$ starting in $y$ going to some vertex $z$ (if $\ell - 5 - a=0$ then $z=y$). 
Such a path exists since $\delta^\pm(H) \geq \frac{n}{4} - o(n)$.
This way, we obtained walks from $x$ to $z$ of lengths $\ell-5$, $\ell-4$, $\ell-3$, and $\ell-2$. From \claref{cla:diam5} there is a path from $z$ to $x$ of length between 2 and 5. In each case, we can find a closed walk of length~$\ell$. 
\end{proof}

Using the proved claims, we can finish the proof of the theorem. 

As noted at the beginning of the proof, $H$ contains a transitive triangle. Denote its source by $u$ and sink by~$v$. Assume there is a directed path $P$ of length at most 4 from $v$ to $u$. 
This path cannot be of length $2$, because there are no directed triangles in $H$. If $P$ is of length $3$, then it creates overlapping cycles of lengths $4$ and $5$, and in consequence a closed walk of length $\ell$ in $H$, because any $\ell \geq 9$ divisible by $3$ can be obtained by adding multiples of $4$ and $5$. Similarly, if $P$ is of length $4$, then it creates overlapping cycles of lengths $5$ and $6$, and so a closed walk of length $\ell$ for all considered here $\ell$ except $\ell=9$. In this missing case, we can use \claref{cla:diam5} for two edges on the constructed $C_5$ (as presented in \figref{fig:l9no4path}) to obtain a contradiction with \claref{cla:1234} (if both paths are of lengths $2$) or a closed walk of length 9 (in all the other possible lengths of the paths).
Therefore, there is no path of length smaller than~$5$ from $v$ to $u$.

\begin{figure}[ht]
{\hfill
\includegraphics{figures/figure-7}
\hfill}
\caption{Setting used in the proof that there is no path of length smaller than $5$ from the sink of a transitive triangle to its source.}
\label{fig:l9no4path}
\end{figure}

Let $A:=N^+_H(v)$ and $D:=N^-_H(u)$.
Since there are no directed triangles in $H$, if every vertex in $A$ has an outneighbor in $A$, then there exists a path on $4$ vertices in~$A$. This gives a contradiction with \claref{cla:1234}, thus there exists a vertex $a\in A$ without outneighbors in~$A$. 
Similarly, there is a vertex $d\in D$ without inneighbors in $D$. Let $B:=N^+_H(a)$ and $C:=N^-_H(d)$. 
Since there is no path of length at most $4$ from $v$ to $u$, the sets $A$, $B$, $C$, and $D$ are disjoint. Each is of size at least $\delta^\pm(H) = \frac{n}{4}-o(n)$, and so by removing the remaining vertices from~$H$ we keep the semidegree assumption of $H$. 
Notice that an edge from $B$ to $A$ (analogously from $D$ to $C$) is giving a contradiction with \claref{cla:1234} for $x=u$ and $y$ being the sink of the assumed edge, and so there are no such edges. Let $b\in B$ be a vertex without outneighbors in $B$. Since $b$ do not have outneighbors in $A$, $B$ and~$D$ (as there is no path of length $4$ from $v$ to $u$), it is connected with $|C|-o(n)$ vertices in $C$. Similarly, there is a vertex $c \in C$ having $|B|-o(n)$ inneighbors in $B$. Thus, we can remove those $o(n)$ vertices from $H$ and keep the semidegree assumption of~$H$. The obtained setting is depicted in \figref{fig:k4walksetup}. 

\begin{figure}[ht]
{\hfill
\includegraphics{figures/figure-8}
\hfill}
\caption{Partitioning of the graph vertices into the sets $A$, $B$, $C$, and $D$.}
\label{fig:k4walksetup}
\end{figure}

Now notice that if at least two of the sets $A$, $B$, $C$, and $D$ contain some edges, then we obtain a contradiction with \claref{cla:1234} for $\ell \geq 15$ and there is $C_9$ forbidden in the case $\ell=9$. Thus, at least three of the sets $A$, $B$, $C$, and $D$ contain no edges. If~$B$ (or by symmetry $C$) contains no edges, all the vertices in $B$ have $\frac{n}{4} - o(n)$ outneighbors in $C$, thus all but $o(n)$ vertices in $C$ have $\frac{n}{4} - o(n)$ inneighbors in $B$. Remove those $o(n)$ vertices from $H$. If $A$ (or by symmetry $D$) contains no edges, then each vertex in $A$ has $\frac{n}{4} - o(n)$ outneighbors in $B$, and so by removing $o(n)$ vertices from $H$ we can obtain that all vertices in $B$ have $\frac{n}{4} - o(n)$ inneighbors in $A$. There cannot be any edges from $C$ to~$A$, as otherwise we can find a forbidden directed triangle in $H$, which implies that each vertex in $A$ has at least $\frac{n}{4} - o(n)$ inneighbors in $D$. Thus, by removing $o(n)$ vertices from~$H$, we can obtain that all vertices in $D$ have $\frac{n}{4} - o(n)$ outneighbors in $A$. Finally, since $C$ or $D$ contains no edges, we can remove $o(n)$ vertices from $H$ to obtain that all vertices in $C$ have $\frac{n}{4} - o(n)$ outneighbors in~$D$ and all vertices in~$D$ have $\frac{n}{4} - o(n)$ inneighbors in $C$. This implies that there are no edges from $D$ to $B$ and inside any of the sets $A$, $B$, $C$ and~$D$. Hence, by removing the starting transitive triangle, we obtain the sought bipartite graph. 
\end{proof}

\thmref{thm:kbig_walk} was crucial in our proof of \thmref{thm:main} for $k \ge 5$, and analogously we are going to use \thmref{thm:k4walk} for proving \thmref{thm:main} in this section, that is, for cycle lengths $\ell \ge 9$ that are divisible by $3$ but not divisible by $4$.
Fix such a cycle length~$\ell$, and suppose there exists an $n$-vertex graph $G$ contradicting \thmref{thm:main} for an arbitrary large $n$.
As in the previous section, a combination of \propref{prop:removallemma} and \thmref{thm:k4walk} yields that we can remove $o(n)$ vertices and $o(n^2)$ edges from $G$ in order to find a bipartite $H \subseteq G$ with $\delta^\pm(H) \geq \frac{n}4-o(n)$.
Now we show that there are $o(n)$ vertices and $o(n^2)$ edges in $H$ such that their removal yields a subgraph of a blow-up of $C_4$ with the minimum semidegree at least $\frac{n}{4}-o(n)$.

\begin{claim}
$G$ contains a subgraph $H'$ homomorphic to $C_4$ with $\delta^\pm(H') \geq \frac{n}{4} - o(n)$.
\end{claim}

\begin{proof}
Note that it is enough to find such a subgraph $H'$ inside $H$.
Let $(L,R)$ be any bipartition of $H$.
Firstly, we show that there exist disjoint sets $A \subseteq L$ and $D \subseteq R$ of at least $\frac{n}{4}-o(n)$ vertices, without edges in $G$ from~$A$ to $D$.

If $\ell$ is even, then take an arbitrary path $P$ in $H$ of length $\ell-3$ and let $v_L \in L$ and $v_R \in R$ be its source and sink. We set $A:=N^+_H(v_R) \setminus P$ and $D:=N^-_H(v_L)\setminus P$. Clearly, both $A$ and $D$ have sizes at least $\frac{n}{4}-o(n)$, and no edge in $G$ goes from $A$ to $D$ since such an edge would create a~copy of~$C_\ell$.

If $\ell$ is odd and there exists $v \in V(G)\setminus V(H)$, then without loss of generality, $v$ has an inneighbor in one part (say $L$), and an outneighbor in the other one. In the case $V(G) = V(H)$, there must be an edge in~$G$ inside the larger part, which also yields an existence of a vertex $v \in V(G)$ that has inneighbor in $L$ and outneighbor in $R$.
Let $v_L$ be an inneighbor of $v$ in $L$, and $P$ be an arbitrary path in $H$ of length $\ell-5$ with source in some outneighbor of $v$ in $R$ and sink in some $v_R \in R$, that is not going through $v_L$.
Now we set $A:=N^+_H(v_R) \setminus (P \cup \{v_L,v\})$ and $D:=N^-_H(v_L)\setminus (P \cup \{v_R,v\})$.
As before, both $A$ and $D$ have sizes at least $\frac{n}{4}-o(n)$, and no edge in $G$ goes from $A$ to $D$ since such an edge would create a~copy of~$C_\ell$.

Now, let $C:=L \setminus A$ and $B:=R \setminus D$.
Clearly $N_H^+(A) \subseteq B$ and $N_H^-(D) \subseteq C$ and both $B$ and $C$ have sizes at most $\frac{n}4+o(n)$.
Hence, the lower bound on $\delta^\pm(H)$ yields that only $o(n^2)$ edges can go from $B$ to $A$ so there are at least $\frac{n^2}{16} - o(n^2)$ edges of $H$ that go from $B$ to $C$.
Analogously, at least $\frac{n^2}{16} - o(n^2)$ edges of $H$ go from~$C$ to $D$.
Therefore, we have found a subgraph of $G$ with $\frac{n^2}{4} - o(n^2)$ edges that is homomorphic to $C_4$.
Since only $o(n)$ vertices of any such subgraph can have the indegree or the outdegree below $\frac{n}{4} - o(n)$, removing all the low-degree vertices from it yields the sought subgraph $H' \subseteq G$.
\end{proof}

The proof of \thmref{thm:main} for $k=4$ and $\ell \geq 9$ now readily follows from the following:

\begin{lemma}\label{lem:c4blowup}
Fix an integer $\ell\ge 9$ divisible by $3$ but not divisible by $4$, and let $G$ be an $n$-vertex graph with the minimum semidegree
\[\delta^\pm(G) \ge
\begin{cases}
(n+1)/4 & \textrm{when } \ell \equiv 3 \textrm{ mod }4, \\
(n+2)/4 & \textrm{otherwise.}
\end{cases}
\]
If $G$ contains a subgraph $H$ homomorphic to $C_4$ with $\delta^\pm(H) \geq \frac{n}{4}-o(n)$, then $G$ contains~$C_\ell$.
\end{lemma}

\begin{proof}
Suppose for contradiction that $G$ is $C_\ell$-free.
Since $H$ is a subgraph of a blow-up of~$C_4$, we refer to its parts as blobs.
Let us now look what are the edges in $G$ that the vertices of a blob might induce.
Clearly, no blob can contain a path of length $3$.
Also, if there are $3$ disjoint edges in $G$ inside the blobs of $H$, then we find~$C_\ell$ for any considered here~$\ell$.
Indeed, when $\ell=9$, any edge inside a blob creates $C_9$, and for $\ell \ge 15$ we simply walk around the four blobs three-times and pick one edge inside a blob at a time.
Therefore, we can remove at most~3 vertices from $H$ to get rid of all edges inside blobs.
Similarly, any 3 disjoint edges in $G$ between a given pair of the non-consecutive blobs of $H$ creates $C_\ell$, hence removing at most 4 vertices in each blob destroys all such edges.

We refer to the parts of $H$ in their cyclic order as $A$, $B$, $C$, and $D$, and call them blobs.
By the discussion in the previous paragraph, we may assume that $V(H)$ induces a bipartite graph in $G$.
Note that the only edges induced in $G$ by $V(H)$ that are not in $E(H)$ go between consecutive blobs in the reverse order. Moreover, since $\delta^\pm(H) \geq \frac{n}{4}-o(n)$, there are only $o(n^2)$ such edges.

Among all the possible choices of such $H \subseteq G$ we select a one maximizing $|V(H)|$, and let $T:=V(G)\setminus V(H)$.
Recall a sidewalk is a path of length $2$ starting and finishing in $V(H)$ with the middle vertex in $T$,
and the value of a sidewalk is the change in the remainder modulo $4$ of cycle lengths that can be achieved by using this sidewalk. 
Also recall that two sidewalks are compatible if they are either vertex disjoint or they share exactly one vertex that is then the sink of one of the sidewalks and the source of the other one.	
As in the proof of \thmref{thm:main} for $k\ge5$, our aim is to find a compatible set of sidewalks that can be combined to a copy of $C_\ell$. 

We call a \emph{blob-transversal} any set of $4$ vertices of $G$ that contains exactly one vertex from each blob.
For a blob-transversal $S$, let $e(S,T)$ be the number of edges between $S$ and $T$ (regardless of their direction).
Since each vertex of $H$ has neighbors only in the neighboring blobs or in $T$, summing up the sizes of the neighborhoods of a fixed blob-transversal $S$ yields the following estimate:
\begin{equation}\label{eq:sidewalks_bound}
e(S,T) + 2(n-|T|) \geq 8 \cdot \delta^\pm(G).
\end{equation}

If a vertex in $T$ creates a sidewalk of value $\ell$ mod $4$, then we easily find $C_\ell$.
Therefore, there are no sidewalks of such a value.
Let us now restrict adjacencies between the vertices in $T$ and blob-transversals.

\begin{claim}\label{cla:k4:atmost3}
For every $t\in T$ and blob-transversal $S$ there are at most $3$ neighbors of $t$ inside $S$.
\end{claim}

\begin{proof}
Suppose for contrary that there is such a vertex $t \in T$, a blob-transversal~$S$, and $4$ edges between them.
If $t$ has $4$ outneighbors (or inneighbors) in $S$, then pick an inneighbor (or outneighbor) of $t$ outside $S \cup T$.
Such a neighbor together with $t$ and $S$ creates a sidewalk of any value modulo $4$ contradicting that there is no sidewalk of value $\ell$ mod $4$.

Therefore, not all the edges between $t$ and $S$ have the same direction and $S\cup\{t\}$ contains sidewalks of both values $1$ and $3$. 
This readily yields a contradiction unless $\ell \equiv 2$~mod~$4$.
In this case, our aim is to find another sidewalk of a non-zero value that is disjoint with $S \cup \{t\}$.

Fix a blob-transversal $S'$ that is disjoint from $S$, and let $t' \neq t$ be a vertex in $T$ that has at least three neighbors inside $S'$.
If $S' \cup \{t'\}$ contains no sidewalk of a non-zero value, then all the edges between $t'$ and $S'$ must have the same orientation.
However, the vertex~$t'$ must be adjacent to $n/4 - o(n)$ vertices of $H$ using edges with the other orientation, which yields a sidewalk of a non-zero value disjoint from $S \cup \{t\}$.
Since its value cannot be $2$, it must be either $1$ or $3$. In both cases we combine the new sidewalk with the one of the same value in $S\cup\{t\}$ to find a copy of $C_\ell$.
\end{proof}

This claim is already sufficient for ruling out the possibility that $\ell$ is even.

\begin{claim}\label{cla:k4:cong2}
If $\ell \ge 18$ and $\ell \equiv 2$ mod $4$, then $G$ contains $C_\ell$.
\end{claim}
\begin{proof}
By \eqref{eq:sidewalks_bound}, we have $e(S,T) \ge 2|T|+4$ for every blob-transversal $S$.
Together with \claref{cla:k4:atmost3}, it yields that for every blob-transversal $S$ there are at least $4$ vertices in $T$ with exactly $3$ neighbors in $S$.
In particular, every blob-transversal is in $4$ different sidewalks of a non-zero value.
Since these sidewalk values can be only $1$ or $3$, we find two disjoint sidewalks of the same value (simply consider three vertex disjoint blob-transversals).
As $\ell\ge18$, any two disjoint sidewalks of the same value can be combined and create $C_\ell$. 
\end{proof}

For the rest of the proof, we assume $\ell$ is an odd multiple of $3$. 

\begin{claim}\label{cla:k4:noSW2}
There are no sidewalks of value $2$ and no edges between consecutive blobs that do not agree with the cyclic order of the blobs.
\end{claim}

\begin{proof}
Assuming the contrary, there is a directed path $uvw$ with $u$ and $w$ from the same blob. 
Consider a blob-transversal $S$ using an inneighbor of $u$ and an outneighbor of $w$.
From \eqref{eq:sidewalks_bound} and \claref{cla:k4:atmost3} there is a vertex $t \neq v$ in $T$ connected with $3$ vertices from $S$. 
If it creates a sidewalk of a non-zero value, then its value is $1$ or~$3$, and it can be combined with the path $uvw$ to find $C_\ell$ for every $\ell \geq 9$ considered here.
In the other case, all $3$~edges are going from (or to) $t$. But then, with any inneighbor (or outneighbor) of $t$, we have a sidewalk of value~$1$ or $3$. Combining it with the assumed path $uvw$ we find $C_\ell$.
\end{proof}


Let us now partition the vertices in the set $T$ to sets $T_A$, $T_B$, $T_C$, $T_D$, and $T'$ based on adjacencies to the blobs they have.
Recall that any vertex $t \in T$ cannot simultaneously have an inneighbor and an outneighbor inside the same blob by~\claref{cla:k4:noSW2}, and \claref{cla:k4:atmost3} yields that $t$ cannot create sidewalks of value $0$ with two different pairs of blobs.
If $t \in T$ has an inneighbor in $A$ and an outneighbor in $C$, then we put $t$ to the set~$T_B$.
Symmetrically, having an inneighbor and an outneighbor in $B$ and $D$, $C$ and $A$, or $D$ and $B$ results in placing $t \in T$ inside $T_C$, $T_D$ or $T_A$, respectively.
For brevity, we set $A' := A \cup T_A$, $B' := B \cup T_B$, $C' := C \cup T_C$, and $D' := D \cup T_D$, and refer to these sets as~\emph{extended blobs}.
Finally, let $T'$ be the set of vertices in $T$ not creating any sidewalk of value $0$, that is, $T' := T\setminus \left(T_A \cup T_B \cup T_C \cup T_D\right)$.

\begin{claim}\label{cla:k4:cong3}
If $\ell \equiv 3$ mod $4$, then $G$ contains $C_\ell$.
\end{claim}

\begin{proof}
Recall that in this case $\ell \geq 15$, $\delta^\pm(G) \geq \frac{n+1}{4}$ and there are no sidewalks of value 3. 

Suppose first there is an edge between a diagonal pair of extended blobs, say from $b \in B'$ to $d \in D'$.
Since any path of length~$3$ from a vertex in $N^-(b) \cap A$ to a vertex in $N^+(d) \cap A$ can be extended to an $\ell$-cycle, we may assume that $b \in T_B$, $d \in T_D$, $N^+(b) \cap D' \subseteq T_D$, $N^-(d) \cap B' \subseteq T_B$, and crucially there is exactly one vertex $v \in A \cap \left(N^-(b) \cup  N^+(d)\right)$.
However, in this case $b$ must have at least $n/4-o(n)$ inneighbors in $B$ since any inneighbor of $b$ in $C$ or $D$ results in a forbidden sidewalk of value $2$ or $3$, respectively.
Similarly, any outneighbor of $b$ in $A$ contradicts \claref{cla:k4:noSW2}, thus all but $o(n)$ outneighbors of $b$ lie in $C$.
A symmetric argument yields that $d$~has $n/4-o(n)$ outneighbors in $D$ and $n/4-o(n)$ inneighbors in $C$.
Therefore, $G$ contains a path $P$ of length~$4$ from a vertex in $B$ to a vertex in $D$.
Now consider any blob-transversal $S$ disjoint from $P \cup \{v\}$.
Since both $b$ and $d$ can have at most $2$ neighbors in $S$, the estimate \eqref{eq:sidewalks_bound} yields, as in the proof of \claref{cla:k4:noSW2}, existence of a sidewalk of value $1$ disjoint from $P$. Combining these two results in an $\ell$-cycle. Therefore, no edge of $G$ goes between a diagonal pair of the extended blobs.

Now, let us focus on a possible existence of an edge going between two consecutive extended blobs in the reverse direction.
By symmetry, suppose there is an edge from $b\in B'$ to $a \in A'$. By definition of the set $A'$, $a$ has an outneighbor in $B$. Analogously, $b$ has an inneighbor in $A$, so $G$ contains a path~$P$ of length~$3$ from a vertex in $A$ to a vertex in $B$.
Since any edge inside $C'$ or $D'$ can be combined together with $P$ to get a copy of $C_\ell$, we may assume that both $C'$ and $D'$ are independent sets.
Similarly, for any $c \in C$ and $d \in D$, no vertex in $T'$ can be connected with both $c$ and $d$, as otherwise it creates a forbidden sidewalk or a sidewalk of value $1$ that can be combined with $P$ to obtain an $\ell$-cycle. 
Thus, the sets $N^+(c)$, $N^-(c)$, $N^+(d)$, and $N^-(d)$ are disjoint, and their sizes sum up to $4\delta^\pm(G) \geq n+1$, which gives a contradiction. 

\begin{figure}[ht]
{\hfill
\includegraphics{figures/figure-18}
\hfill}
    \caption{A possible structure of $G$ in the proof of Claim~\ref{cla:k4:cong3}: there might be edges inside the extended blobs, but no edges are between the diagonal extended blobs.}
    \label{fig:k4:cong3}
\end{figure}

Figure~\ref{fig:k4:cong3} depicts the only possible structure the graph $G$ might still have. 
Suppose now there is an extended blob, say $A'$, such that the induced subgraph $G[A']$ has positive minimum outdegree. In this case, we find a path of length $3$ inside $A'$ or three disjoint edges inside $A$, and either of these options yields an $\ell$-cycle in $G$.
Thus, for the rest of the proof, let us fix inside every extended blob one vertex that has no outneighbor inside its own extended blob, and let us call the respective vertices $a$, $b$, $c$, and $d$.
Since they can have outneighbors only in the succeeding extended blob and in~$T'$,
summing up their outdegrees yields existence of a vertex $t \in T'$ with at least two inneighbors in $\{a,b,c,d\}$.

By symmetry, we may assume $b$ is an inneighbor of $t$. 
If $t$ has some outneighbor in $B$, then we find a path of length $3$ from $A$ to $B$, which leads to an $\ell$-cycle in the same way as in the proof of \claref{cla:k4:noSW2}. 
If $t$ has at least two outneighbors in $A$, then there is a path of length $3$ from $A$ to $A$, which yields an $\ell$-cycle. 
This implies that $t$ cannot have more that two inneighbors among $\{a,b,c,d\}$, and the two are from consecutive extended blobs.
Without loss of generality, the other inneighbor of $t$ is $a$, hence $t$ must have $n/4 - o(n)$ outneighbors in $C$, and so $n/4 - o(n)$ inneighbors in $B$.
Observe that $a \in T_A$, as otherwise $t \in T_B$ contradicting $t \in T'$.
Therefore, to not contradict the maximality of $H$, $a$~must have an inneighbor in $A$, which creates a path $P$ of length~$4$ from $A$ to $C$ through $B$ and~$t$. 
Since $t$ has no neighbors in $A$ and $D$, so it has neighbors in only two blobs, applying \eqref{eq:sidewalks_bound} for any blob-transversal disjoint from $P$ yields two sidewalks of value $1$ not using $t$. In particular, one of them does not use $a$ and combining this sidewalk with $P$ gives an $\ell$-cycle.
\end{proof}

\begin{claim}\label{cla:k4:cong1}
If $\ell \equiv 1$ mod $4$, then $G$ contains $C_\ell$.
\end{claim}

\begin{proof}
In this case $\ell \geq 9$, $\delta^\pm(G) \geq \frac{n+2}{4}$ and there are no sidewalks of value $1$.
If there exists an edge inside an extended blob, say $B'$, then any outneighbor of its sink in $C$ and any inneighbor of its source in $A$ yields a path of length $3$ from $A$ to $C$. Such a path, however, can be extended to an $\ell$-cycle. 
Therefore, we may assume there are no edges inside the extended blobs; see Figure~\ref{fig:k4:cong1} for a possible structure of the graph.

\begin{figure}[ht]
{\hfill
\includegraphics{figures/figure-17}
\hfill}
\caption{A possible structure of $G$ in the proof of Claim~\ref{cla:k4:cong1}: there might be edges between the diagonal extended blobs, but no edges are inside the extended blobs.}\label{fig:k4:cong1}
\end{figure}

Without loss of generality, let $B'$ be the smallest set among the extended sets. 
In particular, it is smaller than $\frac{n+2}{4}$, and thus every vertex $a \in A$ must have an outneighbor $v$ in $T_C$ or $T'$ (recall $a$ has no outneighbors in $D'$ by \claref{cla:k4:noSW2}). 
Either way, $v$ can have outneighbors only in $T \cup D$, and so it creates sidewalks of value~$3$ with at least $n/4-o(n)$ outneighbors in $D$. 
Analogously, every vertex in $C$ has an inneighbor $w$ in $T_A$ or $T'$ and $w$ yields sidewalks of value $3$ with at least $n/4-o(n)$ inneighbors in $D$.
Observe that no vertex in $T'$ can be in both types of these sidewalks, as otherwise it has both an inneighbor in $A$ and an outneighbor in $C$, which would result in putting it to $T_B$. 

Now if $|B'| < \frac{n-2}{4}$, then every vertex from $A$ and $C$ must participate in two different sidewalks of value $3$.
In particular, there exist three compatible sidewalks that can be used to obtain an $\ell$-cycle. 
On the other hand, if $|B'| \geq \frac{n-2}{4}$, then there is a different extended blob of size smaller than $\frac{n+2}{4}$, which yields existence of another two types of sidewalks of value $3$.
Analogously to the previous case, we find three compatible sidewalks of value~$3$ among these sidewalks and those we have found in the previous paragraph, so $G$ contains an $\ell$-cycle.
\end{proof}

Since Claims~\ref{cla:k4:cong2}, \ref{cla:k4:cong3}, and \ref{cla:k4:cong1} cover all the cycle lengths $\ell \ge 9$ that are divisible by~$3$ but not divisible by~$4$, \lemref{lem:c4blowup} and hence also \thmref{thm:main} for this range of the parameters are both proven.
\end{proof}

\section{Semidegree threshold of $C_6$}\label{sec:l6}

Throughout the whole section, let $G$ be an oriented graph on $n$ vertices with $\delta^\pm(G) \ge \frac{n}{4} + \frac{1}{2}$.
Assume by contradiction that $G$ is $C_6$-free.

We follow similar lines as in~\secref{sec:kbig} and \secref{sec:k4}.
Let us start with an analogue of~\claref{cla:diam5}.

\begin{claim}\label{cla:diam5l6}
For any two vertices $x, y \in V(G)$ and a set $Z \subseteq V(G)$ of size at most $10$, there is a path of length at most $5$ from $x$ to $y$ that is internally vertex disjoint from $Z$.
\end{claim}

\begin{proof}
Let $X_1 := N^+(x)\setminus Z$ and for $i \in \{1,2\}$, let $X_{i+1} := X_i \cup N^+(X_i)\setminus Z$.
Clearly, $|X_1| \ge 0.25n - 10$.

The induced subgraph $G[X_1]$ is $C_6$-free, thus applying \corref{cor:flags} to $G[X_1]$ yields a vertex $x_1\in X_1$ such that $|N^+(x_1)\cap X_1|\leq 0.267 |X_1|$.
Therefore,
\[|X_2|\geq |X_1| + |N^+(x_1)| - 0.267 |X_1| - |Z| > 0.43325 n-20\]
by the semidegree assumption.
Analogously, \corref{cor:flags} applied to $G[X_2]$ yields a vertex $x_2\in X_2$ satisfying $|N^+(x_2)\cap X_2|\leq 0.267 |X_2|$, and hence
\[|X_3|\geq |X_2| + |N^+(x_2)| - 0.267 |X_2| - |Z| > 0.56757225 n-30.\]

Similarly for the inneighborhood of $y$, we have $|Y_1|>0.25n-10$, where $Y_1 := N^-(y)\setminus Z$. 
Applying \corref{cor:flags} to $G$ with all the edges being reversed, we find a set $Y_2\subseteq V(G)\setminus Z$ of size $|Y_2| > 0.43325 n-20$ containing only vertices from which there exist a path to $y$ of length at most $2$. 
Since $|X_3|+|Y_2|>n$, the two sets must intersect.
Therefore, $G$ contains a path of length at most $5$ from $x$ to $y$ that avoids the set $Z$.
\end{proof}

In particular, the previous claim yields that there cannot be a path of length $3$ in the outneighborhood and in the inneighborhood of a given vertex.
Using this, we can actually improve the previous claim. 

\begin{claim}\label{cla:diam4}
For any two vertices $x \in V(G)$ and $y \in V(G)$:
\begin{itemize}
\item if $xy \in E(G)$ then there exists a path of length $2$, $3$, or $4$ from $x$ to $y$, which avoids any fixed vertex, and
\item if $xy \notin E(G)$ then there exists a path of length $2$, $3$, or $4$ from $x$ to $y$, which avoids any fixed $3$ vertices. 
\end{itemize}
\end{claim}

\begin{proof}
For brevity, let $Z$ be the set of all the vertices we want to avoid. We define
$A := N^+(x)\setminus Z$, \hbox{$B := N^+(A\setminus\{y\})\setminus Z$}, $D := N^-(y)\setminus Z$, and $C := N^-(D\setminus\{x\})\setminus Z$.

Since there is no path of length $3$ in $N^+(x)$, there exists a vertex $v \in A\setminus\{y\} \subseteq N^+(x)$ such that $|N^+(v) \cap N^+(x)| \le 1$.
Thus, the set $A\cup B$ has size at least $2\delta^\pm(G)-|Z|-1$. 
Similarly, the set $C\cup D$ has size at least $2\delta^\pm(G)-|Z|-1$. 
If those two sets overlap then there exists the sought path. 

Firstly, suppose that $|A\cup B| + |C\cup D| \ge 4\delta^\pm(G)-2|Z| \ge n+2-2|Z|$ yet the two sets do not overlap.
If $xy \in E(G)$ (hence $|Z| = 1$), then $G$ has at least $n+2-|Z| > n$ vertices, which is not possible.
On the other hand, if $xy \notin E(G)$ (hence $|Z| = 3$ and vertices $x$ and $y$ are not in $A\cup B \cup C \cup D$), then $G$ has $n+4-|Z|>n$ vertices; a contradiction.
We conclude that if $A \cup B$ is disjoint from $C \cup D$, then at least one of the considered sets contains exactly $2\delta^\pm(G)-|Z|-1$ vertices and the other one has size at most $2\delta^\pm(G)-|Z|$.

By symmetry, we may assume $|A\cup B| = 2\delta^\pm(G)-|Z|-1$. 
Since there is no path of length $3$ inside $N^+(x)$ yet $|A \cup B| < 2\delta^\pm(G)-|Z|$, we conclude that every vertex in $A\setminus\{y\}$ has an outneighbor in $A\setminus\{y\}$, so $G[A\setminus\{y\}]$ must be a so-called tiling of $C_3$, that is, $|A\setminus\{y\}|/3$ vertex disjoint directed triangles.
Moreover, every vertex in $A\setminus\{y\}$ must have an edge to almost all the vertices in $B$.
If there is no path of length at least $2$ and at most~$4$ from $x$ to $y$, then every vertex in $B$ needs to be adjacent to almost every vertex in~$C$. 
Since any two vertices in the set $A\setminus\{y\}$ with two inneighbors in $C$ create $C_6$, all but one of the vertices in $A\setminus\{y\}$ must have an edge from almost every vertex in $D\setminus\{x\}$.
Finally, almost every vertex in $C$ is adjacent to almost every vertex in $D$.
Therefore, we can find $C_6$ in $G$ by taking two edges of a directed triangle contained inside $A\setminus\{y\}$ and then walking along the sets $B$, $C$, and $D$.
\end{proof}

\begin{claim}\label{cla:c4chord}
Graph $G$ contain no $C_4$ with a chord. 
\end{claim}

\begin{proof}
Firstly, we show that $G$ has no $4$ vertices spanning $C_4$ with exactly one chord. 
Suppose for contradiction there are 4 vertices $a$, $b$, $c$, and $d$ with edges $ab$, $bc$, $cd$, $da$, and $ac$.
\claref{cla:diam4} yields a path $P$ of length $2$, $3$ or~$4$ from $d$ to $a$ avoiding $c$.
If $P$ has length~$4$ we get $C_6$ with vertex $c$.
If $P$ has length $3$, then it needs to avoid vertex~$b$, so it also creates $C_6$.
Thus, there is a path of length $2$ from $d$ to $a$.
Similarly, the path from $c$ to $d$ avoiding $a$ needs to be of length $2$.
However, those two paths of length $2$ yield to $C_6$.

\begin{figure}[ht]
{\hfill
\includegraphics{figures/figure-9}
\hfill}
\caption{Setting in the proof of \claref{cla:c4chord}.}\label{fig:c4chord}
\end{figure}

Now, suppose there is $C_4$, again with vertices $a$, $b$, $c$, and $d$, that has two chords $ac$ and $bd$.
As before, there needs to be a vertex $e$ and edges $de$ and $ea$.
The only way not to create $C_6$ by considering a path from \claref{cla:diam4} from $c$ to $d$ avoiding $a$ is that the only path of length $3$ from $c$ to $d$ uses both vertices $e$ and $b$, and so there are edges $ce$ and $eb$. This is depicted in \figref{fig:c4chord}.
Now, consider a path from \claref{cla:diam4} from $c$ to $d$ avoiding~$e$.
There are no paths of length $2$ and $3$, and a path of length $4$ avoiding the vertex $a$ creates~$C_6$.
Therefore, $a$ needs to be in 
$N^+(c) \cup N^+(N^+(c)\setminus\{d\})$.
This means that there exists a vertex $f$, different from $d$ and $e$, such that $cf$ and $fa$ are both edges in $G$.
Since the vertices $a$, $b$, $c$ and $f$ create $C_4$ with a chord, the previous paragraph yields there needs to be an edge $bf$ or $fb$.
However, in both cases, we find $C_6$, which finishes the proof.
\end{proof}

\begin{claim}\label{cla:c6transT}
Graph $G$ cannot contain a transitive triangle.
\end{claim}

\begin{proof}
Suppose for contrary there is a transitive triangle with edges $ab$, $bc$, and $ac$.
From \claref{cla:diam4}, there is a path of length $2$, $3$ or $4$ from $c$ to $a$ avoiding $b$.
Since length $4$ creates~$C_6$ and length $2$ creates $C_4$ with a chord, which is forbidden by \claref{cla:c4chord}, it must have length~$3$.
Say its edges are $cd$, $de$, and $ea$.
\claref{cla:c4chord} gives also that there are no other edges between vertices $a$, $c$, $d$, and $e$.

\begin{figure}[ht]
{\hfill
\includegraphics{figures/figure-10}
\hfill}
\caption{Setting in the proof of \claref{cla:c6transT}.}\label{fig:c6transT}
\end{figure}

We assert that the sets $N^-(a)$, $N^+(c)$, $N^-(d)\setminus\{b\}$ and $N^+(e)\setminus\{b\}$, depicted in \figref{fig:c6transT}, are pairwise-disjoint:
\begin{itemize}
\item $N^-(a) \cap N^+(c) = \emptyset$ because it yields $C_4$ with a chord contradicting \claref{cla:c4chord}.
\item If there is $x \in N^-(a) \cap N^-(d)\setminus\{b\}$ then from \claref{cla:c4chord} there is no edge $cx$ and from \claref{cla:diam4} there is a path of length $2$, $3$, or $4$ from $c$ to $x$ avoiding $a$, $b$, and $e$. Path of length $2$ needs to avoid also vertex $d$, so in each case we have $C_6$.
\item $N^-(a) \cap N^+(e)\setminus\{b\} = \emptyset$ because it creates $C_6$.
\item $N^+(c) \cap N^-(d)\setminus\{b\} = \emptyset$ because it creates $C_6$.
\item If there is $x \in N^+(c) \cap N^+(e)\setminus\{b\}$ then from \claref{cla:c4chord} there is no edge $xa$ and from \claref{cla:diam4} there is a path of length $2$, $3$, or $4$ from $x$ to $a$ avoiding $b$, $c$, and $d$. Path of length $2$ also has to avoid vertex $d$, so in each case we have $C_6$.
\item Having $x \in N^-(d)\setminus\{b\} \cap N^+(e)\setminus\{b\}$ from \claref{cla:diam4} we get a path of length $2$, $3$, or $4$ from $d$ to $e$ avoiding vertex $x$. A path of length $2$ together with $x$ creates $C_4$ with a chord contradicting \claref{cla:c4chord}. A path of length $3$ avoids also vertices $a$ and $c$, and so creates $C_6$. A path of length $4$ creates $C_6$ with vertex $x$.
\end{itemize}

Observe that the vertex $b$ cannot be in both sets $N^-(d)$ and $N^+(e)$, because it creates a forbidden $C_4$ with a chord.  
Thus $|N^-(a) \cup N^+(c) \cup N^-(d) \cup N^+(e)| \ge 4\delta^\pm(G) > n$, which finishes the proof of the claim.
\end{proof}

Now, we are ready finish the proof of the case $\ell = 6$. 
For any edge $ab \in E(G)$ consider the sets $N^+(a)$, $N^-(a)$, $N^+(b)$ and $N^-(b)$.
Since \[|N^+(a)|+|N^-(a)|+|N^+(b)|+|N^-(b)| \ge 4\delta^\pm(G) \ge n+2,\] \claref{cla:c6transT} yields that there are at least two directed triangles containing the edge $ab$. 

Pick any directed triangle $xyz$ in $G$, and for each edge $xy$, $yz$, and $zx$ consider a directed triangle different from $xyz$ containing this edge.
Clearly, these three directed triangles must be edge disjoint, which readily yields a copy of $C_6$ in $G$.

\section{Other orientations of cycles}\label{sec:otherorientations}

As noted in \cite{KelKO10}, \conjref{conj:main} (hence our \thmref{thm:main}) yields the asymptotic semidegree threshold for any given cycle with a fixed orientation apart from the directed triangle.
Through this whole section, we write oriented cycle to denote a cycle with an arbitrary orientation of its edges, not necessarily the cyclic orientation as in the directed cycle $C_\ell$.

In order to state the main result of this section, we need to define the so-called cycle-type.
For an oriented cycle $C$, we define the \emph{cycle-type} $t(C)$ to be the absolute value of the difference between the number of edges oriented forwards in~$C$ and the number of edges oriented backwards in~$C$ with respect to the cyclic orientation of $C$.
In particular, the $\ell$-cycle has a cycle-type $\ell$, while cycles of cycle-type $0$ are precisely those orientations for which there is a homomorphism into a directed path of an appropriate length.
Observe that if $t(C) \geq 3$, then the cycle-type is equal to the maximum length of a directed cycle into which there is a homomorphism from $C$.

For an arbitrary oriented cycle $C$ with $t(C)\geq 1$ let $k(C)$ be defined as the smallest integer greater than 2 that does not divide $t(C)$.
Notice that if $C$ is an $\ell$-cycle for $\ell \geq 4$, then $k(C)$ coincides with the definition of $k$ in the previous sections.

We are now ready to state the main result of this section.
\begin{theorem}\label{thm:otherorientations}
For every oriented cycle $C \neq C_3$ and $\eps > 0$ there exists $n_0 =n_0(C,\eps)$ such that the following is true for every oriented graph $G$ on $n \ge n_0$ vertices:
\begin{itemize}
\item
If $t(C)\geq 1$ and $\delta^\pm(G) \geq \frac{n}{k(C)}+\eps n$, then $G$ contains $C$.
\item
If $t(C)=0$ and $\delta^\pm(G) \geq \eps n$, then $G$ contains $C$.  
\end{itemize}
Moreover, the bounds on the semidegree are asymptotically the best possible.
\end{theorem}

In order to prove \thmref{thm:otherorientations} for the case $t(C)=3$, we need an auxiliary lemma.
Let $D$ denote the cycle on $5$ vertices with $t(D)=3$, that is, the $4$-shortcut depicted in \figref{fig:shortcuts}.

\begin{lemma}\label{lem:otheror:k3deg}
Every $n$-vertex oriented graph $G$ with $\delta^\pm(G) \geq \frac{n+1}{4}$ contains $C_3$ or $D$.
\end{lemma}
\begin{proof}
Suppose for contradiction there is an oriented graph $G$ on $n$ vertices with semidegree $\delta^\pm(G) \geq (n+1)/4$, not containing $C_3$ and $D$. For any vertex $v \in V(G)$ both $N^+(v)$ and $N^-(v)$ do not contain a path of length $3$, since otherwise there is $D$ in $G$.
In particular, there is $s \in N^+(v)$ such that $N^+(s)$ is disjoint from $N^+(v)$.
Analogously to the reasoning in \claref{cla:diam4}, we find $C_4$ in $G$ that contains $v$. 

Let $v_1$, $v_2$, $v_3$, and $v_4$ be consecutive vertices of some copy of $C_4$ in $G$. Note that the copy must be induced since $G$ is $C_3$-free.
Also, as $G$ is $D$-free, for $i\in[4]$, the sets $N^+(v_i)$ and $N^-(v_i)$ are disjoint from $N^+(v_{i+1})$ and $N^-(v_{i+1})$, respectively, where the indices are taken modulo $4$.
Similarly, the $C_3$-freeness of $G$ yields that the sets $N^-(v_i)$ and $N^+(v_{i+1})$ are disjoint.

By the bound on the outdegree of $G$, we have $\sum_{i\in[4]} |N^+(v_i)| > n$.
Therefore, there must be a common outneighbor of two non-consecutive vertices of the $C_4$.
Without loss of generality, let $x \in N^+(v_1) \cap N^+(v_3)$.
However, $\delta^\pm(G) \geq (n+1)/4$ yields there exists a vertex $y \in N^+(v_2) \cap N^-(v_3)$ and $v_1,v_2,y,v_3,x$ is a copy of $D$; a~contradiction. 
\end{proof}

With this lemma, we can prove \thmref{thm:otherorientations}.

\begin{proof}[Proof of \thmref{thm:otherorientations}]
Since the case $t(C)\leq 2$ was proven in~\cite{KelKO10}*{Proposition 11}, it remains only to prove the case $t(C) \ge 3$.
Let $\ell:=|V(C)|$ and $G$ be a graph contradicting the statement of the theorem.
From the regularity lemma for directed graphs of Alon and Shapira~\cite{AloS04} one can derive the blow-up lemma for directed graphs (the details can be found for example in \cite{KelKO08}). This gives that if $t(C)\geq 4$ then \thmref{thm:main} yields that $G$ contains an $\ell$-blow-up of $C_{t(C)}$.
In particular, $G$ contains $C$, which finishes the proof in the case $t(C)\geq 4$.

Using the same argument for $t(C)=3$ and $C \neq C_3$, we obtain that $G$ cannot contain $C_3$. Therefore, from \lemref{lem:otheror:k3deg}, $G$ contains $D$. 
Since $\delta^\pm(G) \ge 2\ell+5$, we can greedily extend a copy of $D$ to the graph $D^\ell$ obtained by adding a path of length $\ell$ starting in the sink of $D$ and a path of length $\ell$ ending in the source of $D$. This is depicted in \figref{fig:c6dl}.
By the blow-up lemma, we can assume that $G$ contains not only $D^\ell$, but also an $\ell$-blow-up of $D^\ell$.

\begin{figure}[ht]
{\hfill
\includegraphics{figures/figure-11}
\hfill}
\caption{The graph $D^\ell$ consisting of a copy of $D$ and paths of length $\ell$ starting in the sink of $D$ and ending in the source of $D$.}\label{fig:c6dl}
\end{figure}

In order to find a copy of $C$ in $G$, it remains to prove that $C$ is homomorphic to $D^\ell$. 
Without loss of generality, assume $C$ has more forward edges than backward edges.
Since $t(C)=3$ and $C$ is not a directed triangle, $C$ contains two different vertices $x$ and $y$ such that there are exactly $4$ more forward edges than backward edges in $C$ between $x$ and $y$.

Let $u$ and $v$ be the respective source and sink of $D$ in $D^\ell$,
and $P_u$ and $P_v$ the paths ending in $u$ and starting and $v$, respectively.
Observe that we can homomorphically map the segment of $C$ between $x$ and $y$ by mapping $x$ to $u$, $y$ to $v$ and the other vertices of the segment to $V(D^\ell)$ in such a way that no edge of this segment is mapped to $uv$.
On the other hand, the segment of $C$ between $y$ and $x$ can be mapped to the vertices of $P_u$ and~$P_v$ so that $x$ is mapped $u$, $y$ to $v$ and every edge of this segment is mapped to $E(P_u) \cup E(P_v) \cup \{uv\}$. Combining the mappings of the two segments yields a homomorphism from $C$ to $D^\ell$.

The moreover part of the theorem follows from the fact that balanced blow-ups of a cycle of length $k(C)$ are $C$-free.
\end{proof}

\section{Concluding remarks}\label{sec:outro}

In this paper, for any given $\ell \ge 4$ we have determined the exact value of the semidegree threshold forcing $C_\ell$ in large enough oriented graphs.
The threshold is essentially given by the degrees in a balanced blow-up of $C_k$, where $k$ is the smallest integer greater than~$2$ that does not divide $\ell$.
However, unless $k=3$ or $\ell \equiv 3$ mod $12$, one must add to the term $\frac{n}{k}$ coming from the blow-ups of $C_k$ an extra term $\frac{k-1}{2k}$ to overcome certain number-theoretic flavored perturbations.
What remains is to determine the semidegree threshold of $C_3$, which is a long-standing open problem.
\begin{conjecture}[\cite{Ham82}]\label{conj:BCW:trg}
Any $n$-vertex oriented graph $G$ with the minimum semidegree $\delta^\pm(G) \ge \frac n3$ contains~$C_3$.
\end{conjecture}
Since \conjref{conj:BCW:trg} has the same conjectured set of extremal constructions as the triangle case of the Caccetta--H\"aggkvist conjecture (see, e.g., \cite{Raz13}*{Section 2} for its precise description), it is commonly believed that the difficulty of the two problems is very similar.

Despite the extremal constructions are not exactly the balanced blow-ups of directed cycles, we have established in \thmref{thm:kbig_walk} a stability-type result for the cycle lengths divisible by $12$.
Specifically, every extremal graph in this case is close in the so-called edit distance to a balanced blow-up of $C_k$.
The~\emph{edit distance} of two graphs with the same number of vertices is the number of adjacencies one must alter in one of them to obtain a graph isomorphic to the other one.
Using the terminology of graph limits theory (we refer the reader to the monograph by Lov\'{a}sz~\cite{Lov12} for an excellent exposition to the topic), this translates to the fact that the unique extremal so-called oriented graphon is the one representing the limit of balanced blow-ups of $C_k$.

As we have mentioned in the introduction, when the length of the forbidden cycle is an odd multiple of~$3$, there are infinitely many extremal limits.
Indeed, any Eulerian orientation of a balanced complete bipartite graph is asymptotically extremal in this case.
The fact that these are the only extremal limits follows from \thmref{thm:k4walk}.
However, when forbidding $C_\ell$ for $\ell \equiv 6$ mod $12$ and $\ell \neq 6$, a careful inspection of the arguments presented in \secref{sec:k4} yields the following analogue of \thmref{thm:kbig_walk} for these cycle lengths:

\begin{prop}\label{P:C4stab}
Fix an integer $\ell \ge 18$ that is divisible by $6$ and not divisible by $12$.
Every oriented graph $H$ on $n$ vertices with $\delta^\pm(H) \ge \frac{n}{4} - o(n)$ that contains no closed walk of length~$\ell$ is homomorphic to $C_{4}$.
\end{prop}

In the case when $\ell = 6$, we expect that the extremal graphs are close in the edit distance to a blow-up of~$C_4$ where exactly one of the independent sets might be replaced by a one-way oriented bipartite graph.
Similar to \propref{P:C4stab}, maybe even a stronger statement can be made if one is concerned with closed walks.
\begin{question}\label{Q:l6stab}
Is it true that every oriented graph $H$ on $n$ vertices with $\delta^\pm(H) \ge \frac{n}{4} - o(n)$ that contains no closed walk of length~$6$ is homomorphic to the graph depicted in \figref{fig:C5+}?
\end{question}

\begin{figure}[ht]
{\hfill \includegraphics{figures/figure-16} \hfill}
\caption{The conjectured homomorphic image of all the graphs with no closed walk of length $6$ and asymptotically the largest possible minimum semidegree.}\label{fig:C5+}
\end{figure}


Recently, \thmref{thm:KelKO} was strengthen by Czygrinow, Molla, Nagle, and Oursler~\cite{CzyMNO19}, who proved that it is enough to assume the vertices of $G$ have large outdegree.
\begin{theorem}[\cite{CzyMNO19}]\label{thm:CzyMNO}
For every $\ell \ge 4$ there exists $n_0:=n_0(\ell)$ such that every oriented graph~$G$ on $n\ge n_0$ vertices with $\delta^+(G) \ge \frac{n}{3} + \frac{1}{3}$ contains a directed cycle of length exactly~$\ell$.
\end{theorem}

Motivated by the relations between \thmref{thm:KelKO} and \thmref{thm:CzyMNO} as well as \conjref{conj:BCW:trg} and the triangle-case of the Caccetta--H\"aggkvist conjecture, it might be tempting to conjecture that the outdegree threshold of~$C_\ell$ is the same as the semidegree threshold of $C_\ell$ for every length $\ell$.
However, this is not true and for most cycle lengths the corresponding outdegree threshold is strictly larger than the semidegree one.

\begin{figure}[ht]
{\hfill
\includegraphics{figures/figure-6} \hfill
\includegraphics{figures/figure-5}
\hfill}
\caption{Maneuvers used for the constructions assuming only the minimum outdegree. In both maneuvers, by adding presented edges to an arbitrarily chosen vertex in one blob, one can remove a vertex in the topmost blob.}\label{fig:maneuvers_outdegree}
\end{figure}

For an example, fix an integer $\ell$ divisible by $3$ such that the smallest positive integer~$k$ that does not divide~$\ell$ is odd, and consider a balanced blow-up of a $k$-cycle on $n+k-2$ vertices.
Clearly, its minimum outdegree is equal to $\frac{n}{k} + \frac{k-2}{k}$.
If $\ell<k/2$, then using $2\ell-1$ times the first maneuver and $k-2\ell-1$ times the second maneuver from \figref{fig:maneuvers_outdegree}, we remove $k-2$ vertices without changing the minimum outdegree.
Each use of the first maneuver is changing by $\frac{k+1}{2}$ the remainders modulo~$k$ of the cycle lengths that can be obtained, while each usage of the second maneuver is changing those remainders by $\frac{k-1}{2}$.
There is no $\ell$-cycle, since it is impossible to obtain $\ell$ mod $k$ in such a way.
Similarly, for the case $\ell>k/2$, we can use $2\ell-k-1$ times the first maneuver and $2k-2\ell-1$ times the second maneuver from \figref{fig:maneuvers_outdegree}.
Again, we removed $k-2$ vertices without changing the minimum outdegree. 
For $k\ge8$ even, we can provide a similar construction by considering maneuvers that are changing the remainders modulo~$k$ of any cycle length that can be obtained by $3$ and $k-3$ only.  

As we have just observed, the outdegree and semidegree thresholds of cycles cannot be the same for infinitely many cycle lengths. However, we still conjecture that they differ by less than one.

\begin{conjecture}\label{conj:outdeg}
Fix an integer $\ell \ge 4$ and let $k$ be the smallest integer greater than $2$ that does not divide $\ell$.
There exists $n_0:=n_0(\ell)$ such that every oriented graph $G$ on $n\ge n_0$ vertices with $\delta^+(G) \geq \frac{n}{k} + 1$ contains~$C_\ell$.
\end{conjecture}

\section*{Acknowledgements}

We thank Roman Glebov for our joint preliminary work on the $C_6$ case using a different approach and P\'eter P\'al Pach for suggesting the usage of Kneser's theorem in the proof of \lemref{lem:additive}. We also thank the anonymous referee for their valuable comments,
which greatly improved the presentation of our results.

\end{document}